\documentclass[]{interact}
\usepackage{epstopdf}
\usepackage[caption=false]{subfig}
\usepackage{algorithm,algcompatible}
\usepackage{algpseudocode}

\usepackage{mathtools}

\usepackage{xcolor}

\usepackage{cleveref}
\usepackage{thmtools}
\usepackage[numbers,sort&compress]{natbib}
\bibpunct[, ]{[}{]}{,}{n}{,}{,}
\usepackage{url} 
\theoremstyle{plain}
\def\rr{{\mathbb R}}
\def\rn{{\mathbb{R}^n}}
\newcommand{\ared}{{\rm ared}}
\newcommand{\pred}{{\rm pred}}
\DeclareMathOperator{\dom}{dom}
\DeclareMathOperator{\Prox}{Prox}
\DeclareMathOperator*{\argmin}{\arg\!\min}
\newtheorem{theorem}{Theorem}[section]
\newtheorem{lemma}[theorem]{Lemma}
\newtheorem{corollary}[theorem]{Corollary}
\newtheorem{proposition}[theorem]{Proposition}
\newtheorem{assumption}[theorem]{Assumption}
\theoremstyle{definition}
\newtheorem{definition}[theorem]{Definition}

\theoremstyle{remark}
\newtheorem{remark}{Remark}

\begin{document}


\title{A multilevel proximal trust-region method for nonsmooth optimization with applications}

\author{
\name{Robert Baraldi\textsuperscript{a}, Michael Hinterm\"uller\textsuperscript{b,c} and Qi Wang\textsuperscript{b}\thanks{CONTACT Q. Wang. Email: qi.wang@wias-berlin.de}}
\affil{\textsuperscript{a}Optimization and Uncertainty Quantification, Sandia National Laboratories, USA; \textsuperscript{b}Weierstrass Institute for Applied Analysis and Stochastics, Germany;
\textsuperscript{c}Humboldt-Universit\"at zu Berlin, Germany}
}

\maketitle

\begin{abstract}
    Many large-scale optimization problems arising in science and engineering are naturally defined at multiple levels of discretization or model fidelity. Multilevel methods exploit this hierarchy to accelerate convergence by combining coarse- and fine-level information, a strategy that has proven highly effective in the numerical solution of partial differential equations and related optimization problems.
    It turns out that many applications in PDE-constrained optimization and data science require minimizing the sum of smooth and nonsmooth functions. For example, training neural networks may require minimizing a mean squared error plus an $L^1$-regularization to induce sparsity in the weights. Correspondingly, we introduce a multilevel proximal trust-region method to minimize the sum of a nonconvex, smooth and a convex, nonsmooth function. Exploiting ideas from the multilevel literature allows us to reduce the cost of the step computation, which is a major bottleneck 
    in single level procedures. Our work unifies theory behind the proximal trust-region methods and multilevel recursive strategies.  We prove global convergence of our method in finite dimensional space and provide an efficient nonsmooth subproblem solver. We show the efficiency and robustness of our algorithm by means of numerical examples in PDE constrained optimization and machine-learning. 
\end{abstract}


\section{Introduction}

We develop a multilevel trust-region method for the nonsmooth optimization problem
\begin{equation}\label{obj}
    \text{minimize } F(x)\coloneqq f(x)+\phi(x)\quad\text{over }{x\in\rn},
\end{equation}
where $f:\rn\to\rr$ is continuously differentiable and possibly nonconvex, and $\phi:\rn\to\rr\cup\{+\infty\}$ is proper, lower semi-continuous, convex, and not necessarily (classically) differentiable.   
Such nonsmooth optimization problems arise ubiquitously in scientific machine learning and computational optimization, 
where, for instance, smooth dynamics modeled by $f$ interact with nonsmooth regularization or constraints represented by $\phi$.
Typical examples include inverse problems with sparsity-promoting regularizers (or priors)~\cite{v}, partial differential equation (PDE) constrained optimization with control constraints~\cite{hpuu,ik}, and modern learning paradigms such as physics-informed neural networks (PINNs)~\cite{rpk,kgklpwy,wyp} and operator learning~\cite{ljpzk,kllabsa}. 
While gradient-based and quasi-Newton methods dominate smooth optimization, nonsmooth objectives 
-- such as those involving $\ell_1$-penalties, indicator functions, or composite loss functions -- 
require specialized algorithms capable of handling nondifferentiability and nonsmooth geometry. 
First-order proximal algorithms~\cite{pb,bt,r} and trust-region methods~\cite{cgst} are two leading frameworks for such problems. 

The proximal viewpoint enables a natural decomposition of smooth and nonsmooth components, while the trust-region framework leverages gradient and Hessian information for the smooth part and proximal operations for the nonsmooth part, ensuring global convergence and adaptive step control.

Baraldi and Kouri~\cite{bk} recently introduced an  inexact trust-region algorithm for minimizing \eqref{obj}. 
Their algorithm demonstrated global convergence without sacrificing worst-case complexity bounds, out-performed its contemporaries in various classes of data science and PDE-constrained optimization problems. In particular, we adapt their spectral Cauchy point solver~\cite[Algorithm 5]{bk2} and leave other subproblem solver extensions for future work. 
Related research has also explored proximal quasi-Newton methods~\cite{abo} and nonsmooth trust-region approaches for composite problems~\cite{cgt}, providing theoretical foundations and practical insights into solving large-scale nonsmooth problems efficiently.

Large-scale finite-dimensional optimization problems frequently result from the discretization of infinite-dimensional formulations, such as optimal control problems governed by ordinary or partial differential equations~\cite{pm,hpuu}. 
While standard methods (see, e.g., \cite{MR1972219} for classes of PDE-constrained problems) can solve these at a single discretization level (even in a mesh independent fashion \cite{MR2085262,MR2378147}), they don't leverage the availability of multiple resolution levels, where finer discretizations capture more detailed features of the solution and improve accuracy, at the expense of substantially higher computational cost due to larger system sizes and denser numerical operations. 

Multigrid and multilevel methods (see, e.g., \cite{tos,briggs2000multigrid,hackbusch1985multi}), in general, address this challenge by efficiently coupling coarse and fine discretizations: coarse levels accelerate convergence by damping low-frequency errors, while fine levels refine the solution, leading to algorithms that achieve high accuracy at a fraction of the computational effort when compared to a single-level method operating only on the finest scale. 

In a smooth setting, i.e., when $\phi$ is also (Fr\'echet) differentiable, but yet perhaps allowing for bound constraints, e.g., on controls, \cite{MR1809605,MR2196577,MR2505585} have developed multilevel minimization algorithms in PDE-constrained optimization; see also \cite{MR2822818} for a multigrid method for a class of mathematical programs with equilibrium constraints (MPECs) involving partial differential operators.  These methods were mostly inspired by Achi Brandt's seminal nonlinear multigrid scheme for boundary value problems \cite{MR431719}. Concerning nonsmooth objectives, in connection with total variation regularization in image processing, \cite{MR2805356} is one of the few nonlinear multigrid approaches available in the literature. While these methods prove effective in practice, a global convergence analysis as conducted here has remained elusive, even in the smooth setting.


Motivated by these developments and targeting a broad class of model hierarchies, Gratton et al.~\cite{gst} introduced recursive multilevel trust-region methods for smooth problems, where they solve the unconstrained optimization problem with a twice-continuously differentiable objective function which is bounded below, and they exploit the hierarchical approximations to accelerate convergence.
As stated above, using different levels of discretization for an infinite-dimensional problem has be exploited earlier, as this hierarchy can be instrumental for accelerating computations and improve efficiency by transferring information between coarse and fine levels. In this context, the perhaps simplest approach is to use coarser models to compute approximate solutions which can then be extended to starting points for optimization problem on a fine model (see, for instance,~\cite{gt} for a partitioned quasi-Newton technique). In the language of multigrid methods for numerically solving PDEs, this approach is referred to as full approximation scheme. A versatile approach admitting a more general model hierarchy is given by the space mapping technique; see \cite{MR2178485} for an application in PDE-constrained optimization.

Gross and Krause~\cite{gk} extended the results of \cite{gst} by relaxing assumptions on the objective and the trust-region corrections, proving convergence under
more general conditions. They consider a finite-dimensional and nonconvex minimization problem, where the (at least) $C^1$ objective function is neither assumed to be quadratic nor convex.

These works demonstrate that coarse-grained models could efficiently reduce computational costs by providing informative approximations that accelerate fine-level optimization. 
Furthermore, these ideas have since inspired numerous applications, for instance, multilevel Levenberg–Marquardt algorithms for training a neural network with one-hidden layer to find a discretization-free approximate solution of a PDE~\cite{cgrv}.
However, extending multilevel strategies to nonsmooth settings still remains a very complex task requiring dedicated research efforts. 
The main challenges stem from the lack of gradient coherence across levels, the need for proximal
regularization to handle nondifferentiability, and ensuring consistent model reduction across the hierarchy.

Building on these foundational works, this paper proposes a \emph{multilevel proximal trust-region method} for nonsmooth optimization. 
In fact, we extend the framework of Gratton et al.~\cite{gst,gk} to the nonsmooth composite setting~\eqref{obj}, integrating proximal updates and model reduction principles to achieve robustness and scalability. 
In \cite{bk2}, Baraldi and Kouri expound on various proximal trust-region subproblem solvers that generalize traditional trust-region methods for smooth unconstrained and convex-constrained problems. They introduced many efficient proximal subsolvers, including a simplified spectral proximal gradient solver. We adapt their spectral proximal gradient subsolver (SPG) ~\cite[Algorithm 5]{bk2} in particular for this paper.
Our method leverages proximal subsolvers within a multilevel recursive structure, enabling efficient coarse-to-fine corrections even when the objective includes nonsmooth components.

The main contributions of this work are threefold:
\begin{enumerate}
    \item We formulate a recursive multilevel proximal trust-region algorithm for composite optimization problems, combining hierarchical modeling with proximal regularization.
    \item We provide a global convergence analysis under mild assumptions, extending existing multilevel and nonsmooth trust-region results.
    \item We demonstrate the effectiveness of our method for applications to physics-informed neural networks (PINNs) and PDE-constrained optimal control problems, where nonsmooth control costs and high-dimensional structures pose major computational challenges.
\end{enumerate}

This paper is structured as follows. 
In \Cref{sec:prelim}, we collect some mathematical preliminaries and recall the framework for nonsmooth trust-region methods. 
The recursive multilevel trust-region algorithm for our prototypical nonsmooth optimization problem is described in \Cref{sec:rmta}. 
In \Cref{sec:trsub} we provide an efficient subproblem solver, and we show in \Cref{sec:globalcon} that our overall algorithm is globally convergent. Finally, in \Cref{sec:numerics} we provide the numerical results for applications in PDE-constrained optimization and PINNs.

We make some conventions on notation. 
By $\|\cdot\|$ we denote the Euclidean norm in $\rr^n$, and $\langle\cdot,\cdot\rangle$ represents the inner product in $\rn$. 
Let $\mathcal{L}(\rn)$ be the space of continuous linear operators that map $\rn$ into itself, a Banach space endowed with the usual operator norm
$\|B\|=\sup\{\|Bx\|\,|\,\|x\|\leq 1\}, \,\forall B\in\mathcal{L}(\rn).$
Furthermore, for a positive definite matrix $M$ let
$\|x\|_{M}\coloneqq\sqrt{\langle x,Mx\rangle},$
with the related inner product 
$\langle x,y \rangle_M\coloneqq\langle x,My\rangle,$
and
$\|B\|_{M}=\sup\{\|Bx\|\,|\,\|x\|_M\leq 1\}$ the associated operator norm.

\section{Preliminaries}
\label{sec:prelim}
We briefly collect essentials from convex analysis and trust-region methods.
\subsection{Convex analysis}
Following standard convex analysis \cite{MR274683}, we denote the subdifferential of a proper and convex function $\phi:\rn\to(-\infty,\infty]$ by
$$\partial \phi(x)\coloneqq\{\xi\in\rn:\phi(y)\geq \phi(x)+\langle \xi,y-x \rangle, \quad\forall y\in\rn\},$$
and the effective domain of $\phi$ and $\partial\phi$ by
$$\dom \phi\coloneqq\{x\in\rn:\phi(x)<\infty\} \text{ and } \dom\partial\phi\coloneqq\{x\in\rn:\partial\phi(x)\neq\emptyset\},$$
respectively. Furthermore, for $t>0$ the Moreau envelope and proximal mapping of $\phi$, respectively, are
$$\phi_t(y)\coloneqq\min_{x\in\rn}\left\{\phi(x)+\frac{1}{2t}\|x-y\|^2\right\}\; \text{and}\; 
\Prox_{t\phi}(y)\coloneqq\argmin_{x\in\rn}\left\{\phi(x)+\frac{1}{2t}\|x-y\|^2\right\}.$$ 

For more details on the Moreau envelope and the proximal mapping, we refer the reader, e.g., to~\cite{bc,B}.

\subsection{Model definition and step acceptance}
As stated in the introduction, our algorithmic framework for solving \eqref{obj} builds on \cite{bk} and, hence, our analysis utilizes the correspnding assumptions, which we collect here for the sake of self-containedness.
\begin{assumption} \label{passump}(Conditions on Problem \eqref{obj})
\begin{itemize}
    \item [1.] The function $\phi:\rn\to(-\infty,+\infty]$ is proper, closed and convex.
    \item [2.] The function $f:\rn\to\rr$ is $L$-smooth on $\dom\phi$, i.e., there exists an open set $U\subset\rn$, containing $\dom\phi$, on which $f$ is Fr\'echet differentiable and its gradient $\nabla f$ is Lipschitz continuous with modulus $L>0$.
    \item [3.] The objective function $F\coloneqq f+\phi$ is bounded below on $\dom\phi$ by $\kappa_{lb}\in\rr$.
    \item [4.] The Hessian of $f$ at $x\in\mathbb{R}^n$, i.e., $\nabla^2f(x)$, and its approximation $B(x)$ are uniformly bounded from above by the constant $\kappa_H\geq 1$:
\begin{equation}\label{boundhess}
  1+\max\{\|\nabla^2 f(x)\|,\|B(x)\|\}\leq \kappa_H,  
\end{equation}
for all $x\in\rr^{n}$.
\end{itemize}
\end{assumption}
The requirement \eqref{boundhess} is commonly invoked. Indeed, compare, for instance, Gratton et. al.~\cite{cgst}, where it is employed for recursive multiscale trust-region methods for minimizing a smooth objective, or \cite{bk,cgt} for complexity analysis for other nonsmooth methods.

Within a standard trust-region algorithm, one computes a trial iterate $x_{k+1}$ that approximately solves the trust-region subproblem
\begin{equation}\label{subproblem}
\begin{split}
&\text{minimize }m_k(x)\coloneqq f_k(x)+\phi(x)\quad\text{over }{x\in\rn}\\
&\text{subject to } \|x-x_k\|\leq\Delta_k,
\end{split}
\end{equation}
where $x_k\in\dom\phi$ is the current iterate, $f_k$ is a smooth local model of $f$ around $x_k$, and $\Delta_k>0$ is the trust-region radius. 
A typical choice of $f_k$ in \eqref{subproblem} is the quadratic
\begin{equation}\label{quadratic}
    f_k(x)\coloneqq\frac{1}{2}\langle B_k(x-x_k),x-x_k\rangle+\langle \nabla f(x_k),x-x_k\rangle+f(x_k),
\end{equation}
where $B_k:=B(x_k)\in\mathcal{L}(\rn)$ is self adjoint (i.e., a symmetric matrix) approximating the curvature of $f$ at $x_k$. Typical choices for $B_k$ include the Hessian $\nabla^2 f(x_k)$ or a secant approximation thereof.
In general, we require that the trial iterate $x_{k+1}$ satisfies the trust-region constraint
\begin{subequations}\label{eq:Cauchy}
\begin{equation}\label{tr_constraint}
    \|x_{k+1}-x_{k}\|\leq\kappa_{\rm rad}\Delta_k,
\end{equation}
and the fraction of Cauchy decrease (FCD) condition
\begin{equation}\label{FCD}
    m_k(x_k)-m_k(x_{k+1})\geq\kappa_{\rm fcd}h_k\min\left\{\frac{h_k}{1+\|B_k\|},\Delta_k\right\},
\end{equation}
\end{subequations}
where $\kappa_{\rm rad},\,\kappa_{\rm fcd }>0$ are independent of $k$, and, for $t_0>0$, the stationarity measure is
\begin{equation}\label{hk}
    h_k\coloneqq\frac{1}{t_0}\left\|x_k-\Prox_{t_0\phi}(x_k-t_0\nabla f_k(x_k))\right\|.
\end{equation}

In practice, one selects $\kappa_{\rm rad}=1$.
Note that \eqref{FCD} ensures that $x_{k+1}\in\dom\phi$ as otherwise the left-hand side would be $-\infty$.
Given a trial iterate $x_{k+1}$ that satisfies \eqref{eq:Cauchy}, the traditional trust-region algorithm decides whether or not to accept $x_{k+1}$ based on a ratio of actual and predicted reduction
\begin{equation}\label{rhok}
    \rho_k \coloneqq \frac{\ared_k}{\pred_k}=\frac{F(x_k)-F(x_{k+1})}{m_k(x_k)-m_k(x_{k+1})}.
\end{equation}
If the model is a sufficiently accurate approximation to the objective function, $\rho_k$ will be close to one. Then the step is accepted if $\rho_k$ is larger than a chosen threshold $\eta_1\in(0,1)$ and is rejected otherwise. 
In the first case, the step is called \emph{successful}, and otherwise, the step is \emph{unsuccessful}.
Upon step acceptance, the trust region radius is updated for the next iteration, still based on \eqref{rhok}. Indeed,
if $\rho_k\geq\eta_2$ (this step is called \emph{very successful}), the trust-region radius is increased, and otherwise reduced if $\rho_k<\eta_1$. The algorithmic parameters $0<\eta_1<\eta_2<1$ are user-specified with common values $\eta_1=10^{-4}$ and $\eta_2=0.75$~\cite{bk}. Algorithm \ref{Alg1} summarizes this update scheme.
\begin{algorithm}[H]
\caption{Nonsmooth Trust-Region Algorithm~\cite{bk}\label{Alg1}}
\begin{algorithmic}
\State \textbf{Require:} Initial guess $x_{0}\in\dom\phi$, initial radius $\Delta_{0}>0$, $0<\eta_1<\eta_2<1$, and $0<\gamma_1\leq\gamma_2<1\leq\gamma_3$.
\State \textbf{for} $k=0,1,2,\dots$ \textbf{do}
\State \quad\textbf{Model Selection:} Choose model $m_k$ as in \eqref{subproblem}  and \eqref{quadratic}
\State \quad\textbf{Step Computation:} Compute $x_{k+1}\in\rn$ that satisfies \eqref{eq:Cauchy}
\State \quad\textbf{Step Acceptance and Radius Update:} Compute $\rho_k$ as in \eqref{rhok}
\State \quad\textbf{if} $\rho_{k} < \eta_1$ \textbf{ then} 
\State \quad\quad$x_{x+1}\gets x_{k}$
\State \quad\quad$\Delta_{k+1}\in[\gamma_1\Delta_k,\gamma_2\Delta_k]$ 
\State \quad\textbf{else if} $\rho_k\in[\eta_1,\eta_2)$ \textbf{ then}
\State \quad\quad$\Delta_{k+1}\in[\gamma_2\Delta_k,\Delta_k]$
\State \quad\textbf{else}
\State  \quad\quad$\Delta_{k+1}\in[\Delta_k,\gamma_3\Delta_k]$
\State \quad\textbf{end if}
\State\textbf{end for}
\end{algorithmic}
\end{algorithm}

\section{Recursive multilevel trust-region algorithms}\label{sec:rmta}

A recursive trust-region algorithm extends the classical trust-region framework by constructing and solving a hierarchy of related subproblems at multiple levels of resolution \cite{gst}.
Each level defines an approximate model of the objective, allowing coarse levels to provide computationally inexpensive guidance for fine-level optimization.
Conceptually, the recursive trust-region framework can be interpreted as dynamically selecting which subproblem model to solve at each iteration, thereby balancing model fidelity and computational cost while preserving global convergence guarantees.
Let $r\in\mathbb{N}$ and $n_i\in\mathbb{N}$ for $i\in\{0,\,\dots,\,r\}$, 
and assume that there is a collection of functions $\{F_i\coloneqq f_i+\phi_i\}_{i=0}^{r}$ 
such that each $F_i$ satisfies Assumption \ref{passump} from $\rr^{n_i}$ to $\rr$ with $n_{i}\geq n_{i-1}$. 
The top-level objective corresponds to the 
original problem \eqref{obj}, with $n_r=n$, $F_r(x)=F(x)$,  $f_r(x)=f(x)$ and $\phi_r(x)=\phi(x)$ for all $x\in\rn$. 
The underlying presumption is that $F_i$ is ``more costly'' to minimize than $F_{i-1}$; in our examples, $F_i$ has more variables than $F_{i-1}$. 
 
We aim to construct a level-dependent objective function $L_{i-1}$, based on $F_{i-1}$,
such that its minimization provides a good search direction on level $i$. 
In a two-level strategy, we use $F_{r-1}$ to construct an alternative model 
$L_{r-1}$ for $F_r=F$ in the neighborhood of the current iterate that is cheaper than 
\eqref{subproblem} at level $r$, and to use this alternative model, whenever suitable, 
to define the step in the trust-region algorithm. 
For more than two levels ($r>1$), 
this can be done recursively with the approximation process stopping at level 0,  
where the trust-region subproblem model \eqref{subproblem} is always used.

For $F_{i-1}$ to be useful at all in minimizing $F_i$, there should be some relation between the variables of these two functions. 
We henceforth assume the following.
\begin{assumption}\label{roworthonormal}
    For each level $i=1,\,\dots,\,r$, there exist two full-rank linear operators $R_i^{i-1}:\,\rr^{n_i}\to \rr^{n_{i-1}}$ (the restriction) and $P_{i-1}^{i}:\,\rr^{n_{i-1}}\to \rr^{n_{i}}$ (the prolongation) such that $P_{i-1}^{i}=\sigma_i(R_{i}^{i-1})^\top$, for some fixed constant $\sigma_i>0$. 
    Moreover, we assume that $R_i^{i-1}$ is row-orthonormal, that is,  $R_i^{i-1}(R_i^{i-1})^\top=I_{n_{i-1}}$, where $I_{n_{i-1}}$ is the identity matrix of size $n_{i-1}\times n_{i-1}$. 
\end{assumption}
We assume $\sigma_i=1$ without loss of generality, as this can be directly obtained from the original form by scaling $P_i$.

Throughout the remainder of the paper, our quantities of interest may carry a double subscript $i,\,k$, where the first index, $\{i\, \vert \, i \in \mathbb{N}, 0\leq i\leq r\}$, is the level index, 
and the second one, $k$, refers to the current iteration within level $i$. It is assumed to be reset to $0$ each time level $i$ is entered.
Consider now some iteration $k$ at level $i$ with current iterate $x_{i,k}$. Then,
we first restrict $x_{i,k}$ to define the initial iterate $x_{i-1,0}$ at level $i-1$, that is, $x_{i-1,0}:=R_{i}^{i-1}x_{i,k}$.
Next, suppose that, based on $F_{i-1}$, one decides to use a low-level model $L_{i-1}$ to compute an update step $s_{i-1}\in\mathbb{R}^{n_{i-1}}$.
In order to prove convergence of our subsequent multilevel scheme, 
the first-order behavior of the models $L_{i}$ and $L_{i-1}$ must be coherent in a neighborhood of $x_{i,k}$ and $x_{i-1,0}$, respectively. This leads to the requirement 
\begin{equation}\label{firstoderconsist}
    \partial L_{i-1}(x_{i-1,0})=R_i^{i-1}\nabla  \tilde{L}_i(x_{i,k})+R_i^{i-1}\partial\phi_i(x_{i,k})=R_i^{i-1}\partial L_{i}(x_{i,k}),
\end{equation}
where $\tilde{L}_i$ denotes the smooth part of $L_i$. Here, we use~\cite[Theorem 3.36]{B}, and Assumption \ref{passump}, Assumption \ref{roworthonormal}, as well as~\cite[Theorem 3.43]{B}.
Note that this is a generalization of the smooth setting of \cite{gst}. 
Clearly, this criterion induces a choice in both the smooth (in which we follow the setting of \cite{gst}) and nonsmooth terms of $L_{i-1}$.
With this in mind and given the respective entities on level $i$, we define the low-level model as
\begin{equation}\label{mlow}
    L_{i-1}(x_{i-1,0}+s_{i-1})\coloneqq F_{i-1}(x_{i-1,0}+s_{i-1})+\langle R_i^{i-1}\nabla \tilde{L}_{i}(x_{i,k})-\nabla f_{i-1}(x_{i-1,0}),s_{i-1}\rangle,
\end{equation}
where 
$$F_{i-1}(x_{i-1,0}+s_{i-1})\coloneqq f_{i-1}(x_{i-1,0}+s_{i-1})+\phi_{i-1}(x_{i-1,0}+s_{i-1}).$$ 
By convention we have for $s_r\in\mathbb{R}^{n_r}$ that
$L_r(x_{r,0}+s_r)\coloneqq F_{r}(x_{r,0}+s_r)
=f_r(x_{r,0}+s_r)+\phi_r(x_{r,0}+s_r)$.
We still need to specify a choice of the nonsmooth part on lower levels. This is done next. 
\begin{definition}\label{lowhigh}
 At level $i\in\{1,\,\dots,\,r\}$, given a function $\phi_i:\rr^{n_{i}}\to\rr$ its associated low-level model, denoted by $\phi_{i-1}$, is defined by $\phi_{i-1}(x_{i-1,k})\coloneqq \phi_i(x_{i,k}+(R_i^{i-1})^\top(x_{i-1,k}-x_{i-1,0}))$ with $x_{i-1,0}=R_i^{i-1}x_{i,k}$ and $R_i^{i-1}$ as in Assumption \ref{roworthonormal}.     
\end{definition}
With this definition, Assumption \ref{roworthonormal}, and \cite[Theorem 6.15]{B}, the following relation between the proximal mappings of $\phi_i$ and $\phi_{i-1}$ holds true.
\begin{remark}
    For every level $i\in\{1,\,\dots,\,r\}$, $x\in\rr^{n_i}$, and $\phi_{i}$ as in Definition \ref{lowhigh}, we have for $t>0$:
    \begin{equation}\label{proxlowhign}
        \Prox_{t\phi_{i-1}}(R_{i}^{i-1}x) = R_{i}^{i-1}\Prox_{t\phi_{i}}(x).
    \end{equation}
\end{remark}
    
Concerning the smooth part on lower levels we proceed as follows.
\begin{definition}\label{lfh}
At level $i\in\{1,\,\dots,\,r\}$, given a function $f_{i}:\, \mathbb{R}^{n_i}\to\mathbb{R}$ its associated low-level model, denoted by $f_{i-1}:\mathbb{R}^{n_{i-1}}\to\mathbb{R}$, is defined by $f_{i-1}(y) = f_{i}(T_{i-1}^i(y))$, where $y\coloneqq(y_1,\,y_2,\,\dots,\,y_{n_{i-1}})^\top\in\mathbb{R}^{n_{i-1}}$, and $T_{i-1}^i:\, \mathbb{R}^{n_{i-1}}\to \mathbb{R}^{n_i}$ is the extension-by-zero operator, i.e., $T_{i-1}^i(y) = (y_1,\,y_2,\,\dots, y_{n_{i-1}},0,\,\cdots,\,0)^\top\in \mathbb{R}^{n_i}$.
\end{definition}

First-order modifications like the one in \eqref{mlow} (i.e., the second term of the sum) are typical in multigrid applications 
in the context of the full approximation scheme (FAS) 
(see, for instance,~\cite{F,N}) and smooth multilevel optimization strategies~\cite{gst}.
While other choices of $\phi_{i-1}$ may satisfy \eqref{firstoderconsist} as well, our experiments indicate that this projection up to the fine level works very well.

We continue by describing our multilevel trust-region scheme. In fact, the main task when entering level $i=0,\,\dots,\,r$ is to minimize $L_{i}$ starting from $x_{i,0}$. 
At iteration $k$ of this minimization process and depending on progress towards stationarity, we select either the model $L_{i-1}(x_{i-1,0}+s_{i-1})$ (given by \eqref{mlow}) 
and enter a coarser trust-region level, or the  ``standard'' (Taylor) model
\begin{equation}\label{taylor}
 m_{i,k}(x_{i,k}+s_i)\coloneqq L_{i}(x_{i,k}+s_i)+\langle\nabla\tilde{L}_i(x_{i,k}),s_i\rangle+\frac{1}{2}\langle B_{i,k}s_i,s_i\rangle,  
\end{equation} 
where $B_{i,k}=B_{i,k}(x_{i,k})$ is the approximation of $\nabla^2f_{i}(x_{i,k})$ and  $s_i\in\mathbb{R}^{n_i}$.
Once the model is chosen, 
we then compute a step $s_{i,k}\in\mathbb{R}^{n_i}$ that generates a decrease on this model within a trust region $\{s_{i}\,|\,\|s_{i}\|_i\leq\Delta_{i,k}\}$ for some trust region radius $\Delta_{i,k}>0.$ 
If the model \eqref{taylor} is chosen, then we use the spectral proximal gradient subproblem solver (SPG),  which we
describe later as \Cref{alg:Subalg} in \Cref{sec:trsub}, to compute $s_{i,k}$. 

The decrease of the model $m_{i,k}$ must satisfy the usual FCD \eqref{FCD}, i.e.
\begin{equation}\label{taylorfcd}
    m_{i,k}(x_{i,k})-m_{i,k}(x_{i,k}+s_{i,k})\geq \kappa_{\rm fcd }h_{i,k}\min\left\{\frac{h_{i,k}}{1+\|B_{i,k}\|_i},\Delta_{i,k}\right\},
\end{equation}
where $B_{i,k}$ is an approximation of $\nabla^2 f_{i}(x_{i,k})$ and
$$h_{i,k}\coloneqq \frac{1}{t}\left\|x_{i,k}-\Prox_{t\phi_i}(x_{i,k}-t\nabla \tilde{L}_i(x_{i,k}))\right\|_i,$$
for fixed $t>0$, as well as 
\begin{equation}\label{steptr}
\|s_{i,k}\|_i\leq \Delta_{i,k}.   
\end{equation}
We note that the norm $\|\cdot\|_i$ in the last expression is level-dependent and defined by
\begin{equation}\label{inorm}
    \|s_i\|_i\coloneqq \sqrt{\langle s_i,\,M_is_i\rangle}\eqqcolon \|s_i\|_{M_i},
\end{equation}
where 
\begin{equation}\label{Mi}
M_i=
    \begin{cases}
        (P_{r-1}^{r}\cdots P_{i+1}^{i+2}P_{i}^{i+1})^\top(P_{r-1}^{r}\cdots P_{i+1}^{i+2}P_{i}^{i+1}), & \text{if } i=0,\,\dots,\,r-1,\\
        I_n, & \text{if } i=r.
    \end{cases}   
\end{equation}
\begin{remark}\label{reinorm}
    Under Assumption \ref{roworthonormal} and for
    $\sigma_i=1$, we have for any $i\in\{1,\,\dots,\,r\}$ that $\|s_i\|_i=\|s_i\|$, 
    where $\|\cdot\|$ is the Euclidean norm in $\rr^{n_i}$. 
    But, in general, $$\|s_i\|_i=\begin{cases}
        \sigma_r\cdots\sigma_{i+1}\|s_i\|, & \text{if } i=0,\,\dots,\,r-1,\\
        \|s_i\|, & \text{if } i=r.
    \end{cases}   $$
\end{remark}
When the coarse model $L_{i-1}$ is chosen for computing an update step, one (approximately) solves
\begin{equation}
\label{eq:lower_level}
\begin{split}
    &\text{minimize } L_{i-1}(x_{i-1,0}+s_{i-1})\quad\text{over }s_{i-1}\in\mathbb{R}^{n_{i-1}},\\
    &\text{subject to }\|s_{i-1}\|_{i-1}\leq\Delta_{i,k}.
\end{split}
\end{equation}
Suppose that this produces a new point $x_{i-1,*}$ such that $L_{i-1}(x_{i-1,*})<L_{i-1}(x_{i-1,0})$ and a corresponding step $s_{i-1}=x_{i-1,*}-x_{i-1,0}$ which must then be brought back to level $i$ by the prolongation $P_{i-1}^{i}$ satisfying both \eqref{inorm} and \eqref{Mi}. Then,
for the trust-region constraint at level $i-1$ we observe that 
\begin{equation}\label{irelation}
    \|s_i\|_i=\|s_i\|_{M_i}=\|P_{i-1}^{i}s_{i-1}\|_{M_i}=\|s_{i-1}\|_{(P_{i-1}^{i})^\top M_iP_{i-1}^{i}}=\|s_{i-1}\|_{M_{i-1}}=\|s_{i-1}\|_{i-1},
\end{equation}
which
implies 
\begin{equation}\label{lowradius}
    \|x_{i-1,*}-x_{i-1,0}\|_{i-1}=\|s_{i-1}\|_{i-1}=\|s_i\|_i\leq\Delta_{i,k}.
\end{equation}
Note that it is not always possible to use the lower level model. For example, it may happen that
$x_{i,k}$ is not a local minimizer of $L_i$ but $R_i^{i-1}x_{i,k}$ is a local minimizer of $L_{i-1}$.
Consequently, as in \cite{gst}, we select $L_{i-1}$ or $m_{i,k}$ based on initial lower-level stationarity criteria.
In particular, the model $L_{i-1}$ is potentially useful only if 
\begin{align*}
    h_{i-1,0}&=t^{-1}\|x_{i-1,0}-\Prox_{t\phi_{i-1}}(x_{i-1,0}-tR_{i}^{i-1}\nabla \tilde{L}_i(x_{i,k}))\|_{i-1}\\
    &= t^{-1}\|R_{i}^{i-1}(x_{i,k}-\Prox_{t\phi_{i}}(x_{i,k}-t\nabla \tilde{L}_i(x_{i,k})))\|_{i-1}
\end{align*}
is large enough compared to $h_{i,k}$. Here, $t>0$ is fixed as before. This relationship follows from the proof of~\cite[Lemma 5]{bk}, \Cref{lowhigh}, \Cref{roworthonormal}, and~\cite[Theorem 6.15]{B}.
We therefore restrict the use of the model $L_{i-1}$ to iterations where
\begin{equation}\label{fail}
    h_{i-1,0}\geq \kappa_{\rm stop} h_{i,k} \quad\text{and}\quad h_{i-1,0}\geq \epsilon_{i-1},
\end{equation}
i.e., 
\begin{equation*}
\begin{aligned}
  &\bigl\|
    R_{i}^{\,i-1}\!\bigl(
        x_{i,k}
        - \Prox_{t\phi_i}\!\left(
            x_{i,k} - t\nabla \tilde{L}_i(x_{i,k})
        \right)
    \bigr)
  \bigr\|_{i-1}\\
  &\quad\;\ge\;
  \max\left\{\kappa_{\mathrm{stop}}\,
  \bigl\|
      x_{i,k}
      - \Prox_{t\phi_i}\!\left(
          x_{i,k} - t\nabla \tilde{L}_i(x_{i,k})
      \right)
  \bigr\|_{i} , t\,\epsilon_{i-1}\right\},
\end{aligned}
\end{equation*}
%
for some constant $\kappa_{\rm stop}\in (0,\min\{1,\min_{i}\|R_{i}^{i-1}\|\})$ 
and where $\epsilon_{i-1}\in (0,1)$ is a user-given tolerance for the first-order criticality measure for $L_{i-1}$. 
Note that, given $\nabla \tilde{L}_i(x_{i,k})$ and $R_i^{i-1}$, this condition is easy to check before even attempting to compute a step at level $i-1$.

\begin{remark}
We note that the criteria in \eqref{fail} are motivated by~\cite{gst}, but here modified via Assumption \ref{roworthonormal} with the constant $\kappa_{\rm stop}\in(0,1)$ and $R_i^{i-1}$, a row-orthonormal matrix.
\end{remark}

\begin{algorithm}
\caption{RMNTR($i$, $L_i$, $x_{i,0}$, $\phi_{i}$, $\Delta_{i+1}$, $\epsilon_i^h$, $\epsilon_i^{\Delta}$, $\Delta_i^s$)}
\label{RMNTR}
\begin{algorithmic}[1]
\State \textbf{Step 0: Require.} Initial guess $x_{i,0}\in\dom\phi_i$, initial radius $\Delta_{i,0}=\min\{\Delta_i^s,\Delta_{i+1}\}$, $0<\eta_1<\eta_2<1$, and $0<\gamma_1\leq\gamma_2<1\leq\gamma_3$ and $k=0$
\State \textbf{Step 1: Model choice.} If $i = 0$ or if \eqref{fail}  fails, go to Step 3; otherwise continue with the (recursive) Step 2. 
\State \textbf{Step 2: Recursive step computation.} 
Set $x_{i-1,0}=R_{i}^{i-1}x_{i,k}$ and let $L_{i-1}$ as in \eqref{mlow}. Choose $\epsilon_{i-1}$ and 
call Algorithm RMNTR($i - 1$, $L_{i-1}$, $x_{i-1,0}$, $\phi_{i-1}$, $\Delta_{i,k}$, $\epsilon_{i-1}^h$, $\epsilon_{i-1}^{\Delta}$, $\Delta_{i-1}^s$),
yielding an approximate solution $x_{i-1,*}$ of $L_{i-1}$ . Then define $s_{i,k} =  P_{i-1}^{i}(x_{i-1,*} - x_{i-1,0})$ and  
\[
m_{i,k}(x_{i,k}+s_{i}) = L_{i-1}(x_{i-1,0}+s_{i-1})
\]  for all $s_{i}=P_{i-1}^{i}s_{i-1}$. Go to Step 4.
\State \textbf{Step 3: Taylor step computation.} Find a step $s_{i,k}$ such that \eqref{taylorfcd} and \eqref{steptr} hold for $m_{i,k}(x_{i,k}+s_{i})$ defined in \eqref{taylor}. Go to Step 4.
\State \textbf{Step 4: Acceptance of the trial point.} Compute 
\[
\rho_{i,k} = \frac{L_i(x_{i,k}) - L_i(x_{i,k}+s_{i,k})}{m_{i,k}(x_{i,k})-m_{i,k}(x_{i,k}+s_{i,k})}.
\]
If $\rho_{i,k} \geq \eta_1$, then set $x_{i,k+1} = x_{i,k}+s_{i,k}$; otherwise let $x_{i,k+1} = x_{i,k}$.
\State \textbf{Step 5: Termination.} Compute $h_{i,k}$. If $h_{i,k} \leq \epsilon_i^h$ or $\|x_{i,k+1} - x_{i,0}\|_i > (1 - \epsilon_i^{\Delta})\Delta_{i+1}$, then return with the approximate solution $x_{i,*} = x_{i,k+1}$.

\State \textbf{Step 6: Trust-region radius update.} Set
\[
\Delta_{i,k}^+ \in \begin{cases}
[\Delta_{i,k}, \gamma_3\Delta_{i,k}] & \text{if } \rho_{i,k} \geq \eta_2, \\
[\gamma_2 \Delta_{i,k}, \Delta_{i,k}] & \text{if } \rho_{i,k} \in [\eta_1, \eta_2), \\
[\gamma_1 \Delta_{i,k}, \gamma_2 \Delta_{i,k}] & \text{if } \rho_{i,k} < \eta_1,
\end{cases}
\]
and
\begin{equation}\label{deres}
 \Delta_{i,k+1} = \min \left\{\Delta_{i,k}^+, \Delta_{i+1} - \|x_{i,k+1} - x_{i,0}\|_i\right\}.   
\end{equation}
Set $k:=k+1$ and return to Step 1.
\end{algorithmic}
\end{algorithm}

\Cref{RMNTR} describes the recursive, multilevel, nonsmooth trust-region (RMNTR) algorithm. 
It is assumed that the prolongations $P_{i-1}^{i}$ and restrictions $R_{i}^{i-1}$ are known and satisfy Assumption \ref{roworthonormal}. 
An initial trust-region radius for each level $\Delta_i^s>0$ is also defined, as well as level-dependent proximal gradient norm tolerances $\epsilon_i^{h}\in(0,1)$ and the trust-region tolerances $\epsilon^{\Delta}_i\in(0,1)$ for $i=0,\,\dots,\,r$. 
The initial data of this algorithm consist of the level index $i$ ($0\leq i\leq r$), a starting point $x_{i,0}$, the nonsmooth part $\phi_i$ of the objective function in this level, the radius $\Delta_{i+1}$ of the level-$(i+1)$ trust region, and the tolerances $\epsilon_i^h$ and $\epsilon_i^{\Delta}$. 
The original task of minimizing $F(x)=F_r(x)=f_r(x)+\phi_r(x)=f(x)+\phi(x)$ is achieved by calling RMNTR($r,\,F_r,\,x_{r,0},\,\phi_r,\,\Delta_{r+1},\,\epsilon_i^h,\,\epsilon_{i}^{\Delta},\,\Delta_{r}^{s})$ for some starting point $x_{r,0}$ and initial trust-region radius $\Delta_{r}^s$, and where we define $\Delta_{r+1,0}=\infty$. 
We also point out that the definition of $m_{i,k}$ in Step 2 is only used for the calculation of $\rho_{i,k}$, which means actually, we solve the subproblem in the lower level models in Step 2. 
The motivation for \eqref{deres} in Step 6 of the algorithm and the termination test $\|x_{i,k+1}-x_{i,0}\|>(1-\epsilon_i^{\Delta})\Delta_{i+1}$ in Step 5 are to guarantee that iterates at a lower level in a recursion remain in the trust region defined at the calling level (see~\cite[Lemma 4.1]{gst} for details).

\section{Trust-region Subproblem Solvers}\label{sec:trsub}
The methods we propose are recursive procedures, so it suffices to concentrate on the two-level case. 
Thus, for the sake of simplicity, from now on we assume that we have just two approximations to our objective function $F$ at our disposal, i.e., $r=1$. 
To unburden our notation, below we write $R$ for $R_i^{i-1}$ and $P$ for $P_{i-1}^{i}$, respectively. This slight abuse of notation is adopted solely for notational convenience and does not affect the generality of the discussion. The complete formulation and proof of the multilevel (i.e., $r>1$) case are provided in Appendix \ref{appendix}.

In iteration $k$, we (approximately) minimize either the Taylor model \eqref{taylor} or the lower level model \eqref{mlow}; compare steps 2 and 3 in Algorithm RNMTR above. 
If we deal with the Taylor model, then all pertinent FCD and feasibility results can be found in ~\cite[Section 5, Alg. 3]{bk2}. 
Hence, here we only prove that the lower-level model selection satisfies FCD~\eqref{taylorfcd} and trust-region feasibility.

Motivated by ~\cite{bk2}, we now show that it is possible to attain Cauchy decrease via a spectral proximal method for minimizing $L_{i-1}$.

\subsection{Spectral Cauchy Points}
The ``spectral'' Cauchy point in iteration $0$ on level $i-1$ of \Cref{RMNTR} is  
\begin{equation}
    x_{i-1,0}^{c}\coloneqq x_{i-1,0}+\alpha_{i-1}(p_{i-1}(t_{i-1})-x_{i-1,0}),
\end{equation}
for $\alpha_{i-1}\in[0,1]$ and $t_{i-1}\in[t_{\min},t_{\max}]$. 
Here, $p_{i-1}(t_{i-1})$ is taken from the proximal gradient path, i.e.,
\begin{equation}\label{pi1}
   p_{i-1}(t_{i-1})\coloneqq \Prox_{t_{i-1}\phi_{i-1}}(x_{i-1,0}-t_{i-1}R\nabla \tilde{L}_i(x_{i,k})),
\end{equation}
and $0<t_{\min}\leq t_{\max}<+\infty$ are user-specified parameters.
By~\cite[Proposition 2]{bk2}, there exists $\alpha_{i-1}\in[0,1]$ which is the minimizer of the quadratic optimization problem
\begin{equation}\label{qi_1}
\begin{aligned}
  \min_{\alpha \in [0,\alpha_{i-1}^{\max}]} \;
  q_{i-1,0}(\alpha)
  \coloneqq \;&
  \frac{\alpha^{2}}{2}
    \langle B_{i-1,0}s_{i-1}^c(t),\, s_{i-1}^c(t)\rangle \\
  &\quad
  + \alpha\Big(
      \phi_{i-1}(x_{i-1,0}+s_{i-1}^c(t))
      - \phi_{i-1}(x_{i-1,0}) \\
  &\qquad\qquad
      {}+ \langle R\nabla \tilde{L}_i(x_{i,k}),\, s_{i-1}^c(t) \rangle
    \Big),
\end{aligned}
\end{equation}
where $s_{i-1}^c(t)\coloneqq p_{i-1}(t)-x_{i-1,0}$ for $t\in[t_{\min},t_{\max}]$ and, in a slight misuse of notation, $s_{i-1}^c:=s^c_{i-1}(t_{i-1})$.
Note that for the sake of flexibility we make here the dependence on $t$ explicit, while fixed choices $\{t_i\}_{i=0}^r$, with $t_i\in[t_{\min},t_{\max}]$ for all $i$, are usually taken in the algorithm.

We are now targeting FCD in the two-level setting and start by establishing an auxiliary result concerning a progress estimate for the nonsmooth component of the objective along the prolonged Cauchy direction.
\begin{lemma}\label{phiidrop}
Let $\widetilde{s_i}(t)\coloneqq\Prox_{t\phi_i}(x_{i,k}-t\nabla \tilde{L}_i(x_{i,k}))-x_{i,k}$ for arbitrary $t>0$, and $s_i(t):=R^\top (p_{i-1}(t)-x_{i-1,0})=R^\top s_{i-1}^c(t)$.
Then
\begin{equation}\label{iphidecrease}
    \phi_i(x_{i,k})-\phi_{i}(x_{i,k}+s_i(t))\geq \langle \nabla \tilde{L}_i(x_{i,k}),\widetilde{s_i}(t)\rangle+\frac{1}{t}\|\widetilde{s_i}(t)\|_i^2.
\end{equation}
\end{lemma}
\begin{proof}
Definition \ref{lowhigh}, Assumption \ref{roworthonormal}, and \eqref{proxlowhign} yield for arbitrary $t>0$ that 
\begin{equation}\label{ssw}
\begin{aligned}
    s_{i} (t)
    &=R^\top(\Prox_{t\phi_{i-1}}(R x_{i,k}-tR\nabla \tilde{L}_i(x_{i,k}))-x_{i-1,0})\\
    &=(R^\top R-I) (x_{i,k}-t\nabla \tilde{L}_i(x_{i,k}))+\Prox_{t\phi_{i}}(x_{i,k}-t\nabla \tilde{L}_i(x_{i,k}))-R^\top R x_{i,k}\\
    & =\widetilde{s_i}(t)+t(I-R^\top R)\nabla \tilde{L}_i(x_{i,k}).
    \end{aligned}
\end{equation}
Note that $R s_i(t) = R \widetilde {s_i}(t)$ from \eqref{ssw} and Assumption \ref{roworthonormal}, and further 
$$R^\top R\widetilde{s_i}(t)=R^\top Rs_i(t)=R^\top s_{i-1}^c(t)=s_i(t).$$
By this, Definition \ref{lowhigh}, Assumption \ref{roworthonormal}, \eqref{irelation}, and~\cite[part 1 of Lemma 1]{bk} for the inequality, we have that
\begin{align*}
    \phi_i(x_{i,k}) - \phi_i(x_{i,k}+s_i(t))
    &= \phi_{i-1}(x_{i-1,0}) - \phi_{i-1}(x_{i-1,0}+s_{i-1}^c(t)) \\
    &\ge \frac{1}{t}\,\langle s_{i-1}^c(t) + t R\nabla \tilde{L}_i(x_{i,k}),\; s_{i-1}^c(t) \rangle \\
    &= \langle \nabla \tilde{L}_i(x_{i,k}),\, s_i(t) \rangle
       + \frac{1}{t}\|s_{i-1}^c(t)\|^{2}_{i-1} \\
    &= \langle \nabla \tilde{L}_i(x_{i,k}),\, s_i(t) \rangle
       + \frac{1}{t}\|R\widetilde{s}_i(t)\|_{i-1}^{2} \\
    &= \langle \nabla \tilde{L}_i(x_{i,k}),\, s_i(t) \rangle
       + \frac{1}{t}\langle R^\top R\widetilde{s}_i(t),\; \widetilde{s}_i(t) \rangle \\
    &= \langle \nabla \tilde{L}_i(x_{i,k}),\, s_i(t) \rangle
       + \frac{1}{t}\langle s_i(t),\; \widetilde{s}_i (t)\rangle \\
    &=
      \langle \nabla \tilde{L}_i(x_{i,k}),\; R^\top R\widetilde{s}_i(t) \rangle
      + \frac{1}{t}\|\widetilde{s}_i(t)\|_i^{2} \\
    &\qquad
      + \left\langle (I - R^\top R)\,\nabla \tilde{L}_i(x_{i,k}),\;
        \widetilde{s}_i(t) \right\rangle \\
    &= \langle \nabla \tilde{L}_i(x_{i,k}),\; \widetilde{s}_i (t)\rangle
       + \frac{1}{t}\|\widetilde{s}_i(t)\|_i^{2},
\end{align*}
which completes the proof.
\end{proof}

Next, we establish the FCD for the prolonged Cauchy step.
\begin{proposition}\label{subsolveralpha}
    Let $\alpha_i\in[0,1]$ be the minimizer of the quadratic optimization problem in \eqref{qi_1}, 
    with $\alpha_{i-1}^{\max}\coloneqq \min\left\{1,\Delta_{i,k}\|\widetilde{s_i}(t)\|_i^{-1}\right\}.$
    If $h_{i,k}>0,$ then $s_i(t)=Ps_{i-1}^c(t)=R^\top s_{i-1}^c(t)$ satisfies \eqref{taylorfcd} with $\kappa_{\rm fcd }=\frac{1}{2}\min\{1,t_{\min}\kappa_{\rm stop}^2,\kappa_{\rm stop}^4/(\kappa_H-1)\}$. 
\end{proposition}
\begin{proof}
    We first recall that $h_{i,k} = t^{-1}\|x_{i,k}-\Prox_{t\phi_i}(x_{i,k}-t\nabla \tilde{L}_i(x_{i,k}))\|_i$ for fixed $t>0$. 
    At each iteration $k$ at level $i$, we either minimize (decrease) the Taylor model \eqref{taylor} or the lower level model \eqref{mlow}. If we choose the Taylor model, then the results are covered by~\cite[Proposition 2]{bk2}. Otherwise, we choose the low level model, and it suffices to work with its quadratic approximation since the optimization step is determined entirely by this local approximation. As we only discuss two levels in this section, we get
    from Step 2 in Algorithm 2 and by the convexity of $\phi_i$ that, for $\alpha\in [0,1]$, 
\begin{equation}\label{lowmd}
\begin{aligned}
&m_{i,k}(x_{i,k}) - m_{i,k}(x_{i,k}+\alpha s_i(t)) \\
&\quad =\, f_{i-1,0}(x_{i-1,0}) - f_{i-1,0}(x_{i-1,0}+\alpha s_{i-1}^c(t)) + \phi_{i-1}(x_{i-1,0}) - \phi_{i-1}(x_{i-1,0}+\alpha s^c_{i-1}(t)) \\
&\qquad\qquad - \left\langle R\nabla \tilde{L}_i(x_{i,k}) 
      - \nabla f_{i-1}(x_{i-1,0}),\, \alpha s_{i-1}^c(t) \right\rangle \\
&\quad = -\frac{1}{2}\alpha^2 \langle B_{i-1,0}s_{i-1}^c(t), s_{i-1}^c(t)\rangle
        - \alpha \langle R\nabla \tilde{L}_i(x_{i,k}), s_{i-1}^c(t) \rangle
        + \phi_{i-1}(x_{i-1,0})\\
        &\qquad\qquad - \phi_{i-1}(x_{i-1,0}+\alpha s_{i-1}^c(t)) \\
&\quad \ge -\frac{1}{2}\alpha^2 \langle B_{i-1,0}s_{i-1}^c(t), s_{i-1}^c(t)\rangle
        - \alpha\!\left[
            \langle R\nabla \tilde{L}_i(x_{i,k}), s_{i-1}^c(t) \rangle
            - \big(\phi_{i-1}(x_{i-1,0})\right.\\
            &\qquad \qquad \left. - \phi_{i-1}(x_{i-1,0}+ s_{i-1}(t))\big)
          \right] \\
&\quad = -q_{i-1,0}(\alpha).
\end{aligned}
\end{equation}

    For the ease of notation, we define the following quantities:
    \begin{eqnarray*}\kappa_i(t)&\coloneqq& \langle B_{i-1,0}s_{i-1}^c(t),s_{i-1}^c(t)\rangle,\\ d_i(t)&\coloneqq&\langle \nabla R\tilde{L}_i(x_{i,k}),s_{i-1}^c(t)\rangle-(\phi_{i-1}(x_{i-1,0})-\phi_{i-1}(x_{i-1,0}+ s_{i-1}^c(t))),
    \end{eqnarray*}
and note that if $\kappa_i(t)>0$, then the unconstrained minimizer of $q_{i-1}$ is given by $-d_i(t)/\kappa_i(t)$. By \cite[part 1 of Lemma 1]{bk}, , the first inequality of \eqref{fail}, and the fact that $t>0$, we have that
\begin{equation*}
 d_i(t)\leq -\frac{1}{t}\|s_{i-1}^c(t)\|^2=-th_{i-1,0}^2\leq-t\kappa_{\rm stop}^2h_{i,k}^2,   
\end{equation*}
 In this case, we have $\alpha_i=\min\{-d_i(t)/\kappa_i(t),\alpha_{i,\max}\}$. When $\kappa_i(t)=0$, then we have $q_{i-1}(\alpha)=\alpha d_i(t)\leq -\alpha t\kappa_{\rm stop}^2h_{i,k}^2$. Therefore, $\alpha_i=\alpha_{i,\max}>0$. Finally, if $\kappa_i(t)<0$, then $q_{i-1}$ is concave and $\alpha_i$ is either 0 or $\alpha_{i,\max}$. Considering the two cases that define $\alpha_{i,\max}$, we obtain that
$q_i(\alpha_{i,\max})\leq -h_{i,k}\min\{t\kappa_{\rm stop}^2h_{i,k},\Delta_{i,k}\}<0=q_i(0)$
and hence, $\alpha_i=\alpha_{i,\max}.$ This demonstrates that there are three cases which we must discuss: $\alpha_i=1$, $\alpha_i=\Delta_{i,k}/\|\widetilde{s_i}(t)\|_i$, and $\alpha_i=-d_i(t)/\kappa_i(t)$. 

For the first two cases $\alpha_i = 1$ and $\alpha_i = \Delta_{i,k} / \|\widetilde{s_i}\|_i$, respectively, the proofs follow the same arguments as in \cite[Proposition~2]{bk2}. 
By replacing $t$, $h_k$, $B_k$, and $s_k$ with $t\kappa_{\rm stop}^2$, $h_{i,k}$, $B_{i,k}$, and $\widetilde{s_i}$, respectively, and repeating the first two case analyses in \cite[Proposition~2]{bk2}, we obtain \eqref{taylorfcd}.\\
\noindent\textbf{Case 3.} $\alpha_i=-d_i(t)/\kappa_i(t)$: In this case, $0<-d_i(t)\leq\kappa_i(t)\leq\|B_{i-1,0}\|_{i-1}\cdot\|s_{i-1}^c(t)\|_{i-1}^2$. Note that $R\widetilde{s_i}(t)=Rs_i(t)=RR^\top s_{i-1}^c(t)=s_{i-1}^c(t)$. Then we have that $\|s_{i-1}^c(t)\|_{i-1}=\|R\widetilde{s_i}(t)\|_{i-1}\leq\|R\|\cdot\|\widetilde{s_i}(t)\|_{i}=\|\widetilde{s_i}(t)\|_{i}$ and
$$m_{i,k}(x_{i,k})-m_{i,k}(x_{i,k}+\alpha_is_i(t))\geq\frac{d_i(t)^2}{2\kappa_i(t)}\geq\frac{t^2\kappa_{\rm stop}^4h_{i,k}^4}{2\|B_{i-1,0}\|_{i-1}\cdot\|\widetilde{s_i}(t)\|_{i}^2}=\frac{\kappa_{\rm stop}^4h_{i,k}^2}{2\|B_{i-1,0}\|_{i-1}}.$$
If $\|B_{i-1,0}\|_{i-1}\leq 1+\|B_{i,k}\|_i$, then it is obvious that 
$$m_{i,k}(x_{i,k})-m_{i,k}(x_{i,k}+\alpha_is_i(t))\geq\frac{1}{2}\cdot\frac{\kappa_{\rm stop}^4h_{i,k}^2}{1+\|B_{i,k}\|_{i}}.$$
Otherwise, by \eqref{boundhess}, we have that $\frac{1+\|B_{i,k}\|_{i}}{\|B_{i-1,0}\|_{i-1}}\geq\frac{1}{\kappa_H-1}$ and hence,
$$m_{i,k}(x_{i,k})-m_{i,k}(x_{i,k}+\alpha_is_i(t))\geq\frac{\kappa_{\rm stop}^4}{2(\kappa_H-1)}\cdot \frac{h_{i,k}^2}{1+\|B_{i,k}\|_{i}}.$$
Combining the cases 1, 2, and 3 proves that \eqref{taylorfcd} holds for $\alpha_i s_i(t)$.
\end{proof}
As a consequence, we may apply the spectral proximal gradient (SPG) subproblem solver to compute the update step for the lower level model $L_{i-1}$ as in~\cite{bk2} with a small modification. For the convenience of the reader, we present the SPG algorithm here in Algorithm 3. 
We recall that $$h_{i-1,0}=\frac{1}{t}\left\|x_{i-1,\ell}-\Prox_{t\phi_{i-1}}(x_{i-1,\ell}-tR\nabla \tilde{L}_i(x_{i,k}))\right\|_{i-1}.$$
From \eqref{Mi}, Assumption \ref{roworthonormal}, and \eqref{lowradius}, we know that, for each iteration $\ell$ in level $i-1$, 
\begin{equation}\label{lr}
  \|x_{i-1,\ell}-x_{i-1,0}\|_{i-1}\leq\Delta_{i,k}.  
\end{equation}

\begin{algorithm}[hbt!]
\caption{SPG Trust-Region Subproblem Solver~\cite{bk2} }
\label{alg:Subalg}
\begin{algorithmic}[1]
\State\textbf{Require:} Initial guess $x_{i-1,0}$, $f_{i-1,0}=f_{i-1,0}(x_{i-1,0})$, $\phi_{i-1,0}=\phi_{i-1}(x_{i-1,0})$, $m_{i-1,0}=f_{i-1,0}+\phi_{i-1,0}$, $d_{i-1,0}=R\nabla \tilde{L}_i(x_{i,k})$, and an integer \texttt{maxit}, and positive tolerances $\widetilde{\tau}$ and $\tau_{i-1}$, the positive safeguards $t_{\min}\leq t_{\max}$, and $t_{i-1,0}=t\in[t_{\min},t_{\max}]$
\State  Set $\ell=0$
\State \textbf{while} $\ell<$\texttt{maxit} and $h_{i-1,\ell}>\min\{\widetilde{\tau},\tau_{i-1}h_{i-1,0}\}$ and $\|x_{i-1,\ell}-x_{i-1,0}\|_{i-1}\leq\Delta_{i,k}$ do
\State \quad Set $s\gets\Prox_{t_{i-1,\ell}\phi_{i-1}}(x_{i-1,\ell}-t_{i-1,\ell}d_{i-1,\ell})-x_{i-1,\ell}$
\State \quad Set $\alpha_{\max}\gets 1$
\State \quad \textbf{if} $\|x_{i-1,\ell}+s-x_{i-1,0}\|_{i-1}>\Delta_{i,k}$ \textbf{then}
\State\quad\quad Set $\alpha_{\max}>0$ so that $\|x_{i-1,\ell}+\alpha_{\max}s-x_{i-1,0}\|_{i-1}=\Delta_{i,k}$
\State\quad \textbf{end if}
\State \quad Compute $\widehat{\phi}_{i-1,\ell}\gets\phi_{i-1}(x_{i-1,\ell}+s)$, $b\gets B_{i-1,0}s$, and $\kappa\gets\langle b,s\rangle$
\State \quad\textbf{if} $\kappa\leq 0$ \textbf{then}
\State \quad\quad Set $\alpha\gets\alpha_{\max}$
\State \quad\textbf{else}
\State\quad\quad Set $\alpha\gets\min\{\alpha_{\max},-(\langle d_{i-1,\ell},s\rangle+\widehat{\phi}_{i-1,\ell}-\phi_{i-1,\ell})/\kappa\}$
\State\quad\textbf{end if}
\State\quad Set $x_{i-1,\ell+1}\gets x_{i-1,\ell}+\alpha s$, $d_{i-1,\ell+1}\gets d_{i-1,\ell}+\alpha b$ and $\phi_{i-1,\ell+1}\gets\phi_{i-1}(x_{i-1,\ell+1})$
\State\quad\textbf{if} $\kappa\leq 0$ \textbf{then}
\State\quad\quad Set $\bar{t}\gets t/\|d_{i-1,\ell}\|$
\State\quad\textbf{else}
\State\quad\quad Set $\bar{t}\gets\langle s,s\rangle/\kappa$
\State\quad\textbf{end if}
\State\quad Set $t_{i-1,\ell+1}\gets\max\{t_{\min},\min\{t_{\max},\bar{t}\}\}$
\State\quad Set $\ell\gets\ell+1$
\State\textbf{end while}
\State Return $x_{i-1,*}\gets x_{i-1,\ell}$ as the approximate solution
\end{algorithmic}
\end{algorithm}

\section{Global convergence}\label{sec:globalcon}
In this section, we also restrict the discussion to the two-level case, for the same reasons outlined in \Cref{sec:trsub}. Although the proposed algorithm and analysis extend naturally to the full multilevel setting by recursive settings, focusing on two levels simplifies the exposition and suffices to establish the key arguments underlying the global convergence proof.
The extension to the full multilevel case follows by recursion and is provided in the Appendix \ref{appendix}.

Inspired by the convergence theory reported in~\cite{bk}, we prove the global convergence of the proposed methods and we provide a worst-case complexity bound to reach such a point, generalizing the theory proposed in~\cite{bk}. 
To prove the global convergence of Algorithm 2, we first establish two technical lemmas. In the first lemma, we show that a sufficiently small radius $\Delta_{i,k}$ guarantees in iteration $k$ that $\rho_k\geq\eta_2$. As a consequence, the new trust-region radius satisfies $\Delta_{i,k+1}\geq\Delta_{i,k}$.
\begin{lemma}\label{upbound}
 Consider an iteration $(i,k)$ for which $h_{i,k}>0$ and 
 \begin{equation}\label{conD}
     \Delta_{i,k}\leq\kappa_{s}h_{i,k},
 \end{equation}
 then $\rho_k\geq\eta_2$ and $\Delta_{i,k+1}\geq\Delta_{i,k}$,
 where $\kappa_s\coloneqq \kappa_{\rm fcd }(1-\eta_2)/\kappa_H<1$ with $\kappa_{\rm fcd }$ as in Proposition \ref{subsolveralpha} and $\kappa_H$ as in \eqref{boundhess}.
\end{lemma}
\begin{proof}
In iteration $k$ at level $i$ we either minimize (decrease) the Taylor model \eqref{taylor}, or the lower level model \eqref{mlow}.  In both cases, it holds that
\begin{equation}\label{ml}
    m_{i,k}(x_{i,k})-m_{i,k}(x_{i,k}+\alpha s_{i,k})\geq\kappa_{\rm fcd }h_{i,k}\min\left\{\frac{h_{i,k}}{1+\|B_{i,k}\|},\Delta_{i,k}\right\}\geq \kappa_{\rm fcd }h_{i,k}\Delta_{i,k},
\end{equation}
where $\alpha\in [0,1]$ is as in \Cref{sec:trsub}, and we used \eqref{conD} to get the last inequality.
Let us consider the quantity
\begin{equation}
    |1-\rho_{i,k}|=\left|1-\frac{L_i(x_{i,k}) - L_i(x_{i,k}+\alpha s_{i,k})}{m_{i,k}(x_{i,k})-m_{i,k}(x_{i,k}+\alpha s_{i,k})}\right|\eqqcolon \left|1-\frac{\ared_{i,k}}{\pred_{i,k}}\right|
\end{equation}
If, in step $k$, the Taylor model is chosen, then we have by \eqref{taylor} and Taylor's theorem:
\begin{equation}\label{tde}
\begin{aligned}
   |\ared_{i,k}-\pred_{i,k}|
   &= \Bigl|
        f_i(x_{i,k})
      - f_i(x_{i,k}+\alpha s_{i,k}) \\
   &\qquad\quad
      + \left\langle \nabla \tilde{L}_i(x_{i,k}),\, \alpha s_{i,k} \right\rangle
      + \frac{\alpha^{2}}{2}
        \left\langle B_{i,k}s_{i,k},\, s_{i,k} \right\rangle
      \Bigr| \\
   &\le \frac{\alpha^{2}}{2}
      \Bigl|\left\langle
           \left( B_{i,k}-\nabla^{2}f_i(\xi_{i,k}) \right)s_{i,k},
           s_{i,k}
      \right\rangle\Bigr| \\
   &\le \kappa_H\,\Delta_{i,k}^{2},
\end{aligned}
\end{equation}
where $\xi_{i,k}$ is on the line segment $[x_{i,k},x_{i,k}+\alpha s_{i,k}]$ and, in the last step, \eqref{boundhess} is used.
If the lower level model is chosen, then we have by $s_{i,k}=Ps_{i-1}$ and \eqref{lowmd} that
\begin{equation}\label{lde}
\begin{aligned}
   |\ared_{i,k}-\pred_{i,k}|
   &=\Bigl|
      f_i(x_{i,k}) - f_i(x_{i,k}+\alpha s_{i,k}) \\
   &\qquad\quad
      + \left\langle \nabla \tilde{L}_i(x_{i,k}),\, \alpha s_{i,k} \right\rangle
      + \frac{\alpha^{2}}{2}
        \left\langle B_{i-1,0}s_{i-1,0},\, s_{i-1,0} \right\rangle
     \Bigr| \\
   &\le \frac{\alpha^{2}}{2}\Bigl(
       |\langle B_{i-1,0}s_{i-1}, s_{i-1}\rangle|
       + |\langle \nabla^{2} f_i(\xi_{i,k}) s_{i,k},\, s_{i,k}\rangle|
     \Bigr) \\
   &\le \kappa_H\, \Delta_{i,k}^{2}.
\end{aligned}
\end{equation}
where, in the last step, we used \eqref{lr}.
Combining \eqref{tde}, \eqref{lde}, \eqref{ml}, and \eqref{conD} gives us
\begin{equation}
    |1-\rho_{i,k}|\leq\frac{\kappa_H\Delta_{i,k}^2}{\kappa_{\rm fcd }h_{i,k}\Delta_{i,k}}\leq 1-\eta_2,
\end{equation}
then $\rho_{i,k}\geq\eta_2$.
\end{proof}
This result guarantees the finiteness of the recursion at iteration $(i,k)$ whenever the trust-region radius $\Delta_{i,k}$ is sufficiently small. It also implies the following useful consequence.
\begin{corollary}\label{eachonesuccess}
Each minimization sequence contains at least one successful iteration.
\end{corollary}
\begin{proof}
    Using Lemma \ref{upbound} and repeating the proof of \cite[Lemma 4.6]{gst}, we then finish the proof of Corollary \ref{eachonesuccess}.
\end{proof}

In the following lemma, we establish a constant lower bound on $\Delta_{i,k}$ under the assumption that $h_{i,k}>\epsilon>0$.

\begin{lemma}\label{lowerbound}
    Let $\epsilon>0$ be fixed and suppose $h_{i,k}>\epsilon$ for all $(i,k)$. Let $k_1$ be the index of the first successful iteration at level $i$. Then
    \begin{equation}\label{lowbound}
        \Delta_{i,k}\geq\min\{\gamma_1^{k_1}\Delta_{i,0},\gamma_1\kappa_s\epsilon,\epsilon_{i}^{\Delta}\Delta_{i+1}\}\eqqcolon\Delta_{\min}
    \end{equation}
for every iteration $(i,k)$.
\end{lemma}
\begin{proof}
We discuss only the two-level case. The multilevel case follows then from the recursive property of Algorithm 2.

If $i=r$, then by $\Delta_{r+1}=\infty$, we have $\Delta_{r,k+1}=\Delta_{i,k+1}^{+}$ from \eqref{deres}. Hence, 
$\Delta_{i,k}\geq \gamma_1^{k_1}\Delta_{i,0}$,
for any $k\leq k_1$, and \eqref{lowbound} holds true. 
For $k>k_1$, we consider two cases: either $\Delta_{i,k_1}\geq\gamma_1\kappa_s\epsilon$ or $\Delta_{i,k_1}<\gamma_1\kappa_s\epsilon$. In the first case, if $\rho_{i,k_1}\geq\eta_2$, then \eqref{lowbound} holds for $k=k_1+1$. If $\rho_{i,k_1}<\eta_2$, then by Lemma \ref{upbound}, we have that $\Delta_{i,k_1}>\kappa_sh_{i,k_1}$. Combining this, Step 6 in Algorithm 2 and the fact that $\gamma_2\geq\gamma_1$ and $h_{i,k}>\epsilon$ gives us $\Delta_{i,k_1+1}>\gamma_2\Delta_{i,k_1}>\gamma_1\kappa_s\epsilon,$ i.e. \eqref{lowbound} holds for $k=k_1+1.$  We now consider the second case: we observe that if $\Delta_{i,k}\leq\kappa_s\epsilon$, then by Lemma \ref{upbound} and the fact that $h_{i,k}>\epsilon$, we have that $\rho_{k_1}\geq\eta_2$ and thus, $\Delta_{i,k_1+1}\geq\Delta_{i,k_1}\geq\gamma^{k_1}_1\Delta_{i,0}.$  These arguments could be repeated for all $k>k_1$, which completes the proof for $i=r$.

If $i=r-1$ (i.e. $i=0$ since we only consider two levels), note that, for any $k$, $\|x_{i,k+1}-x_{i,0}\|\leq(1-\epsilon_i^{\Delta})\Delta_{i+1}$, which implies that $\Delta_{i,k+1}=\min\{\Delta_{i,k+1}^{+},\epsilon_{i}^{\Delta}\Delta_{i+1}\}$. Then proof is done by repeating the above proof with $\Delta_{i,k+1}^{+}$.
\end{proof}
We consider the sequence of successful iterations ($\rho_{i,k}\geq \eta_1$). They are divided into two groups: $K_{s,f}$ the successful iterations at which the fine model (Taylor model) has been employed, and $K_{s,l}$ the ones at which the lower level model has been employed. Let $K_s$ be all the successful iterations, i.e. $K_s\coloneqq K_{s,f}\cup K_{s,l}$, and let $k_1$ be the index of the first successful iteration. 
\begin{theorem}\label{finiteite}
The number of iterations at each level is finite. 
\end{theorem}
\begin{proof}
Let $k_1$ be the first successful iteration in level $i$. Then by Proposition \ref{subsolveralpha}, the bound in \eqref{boundhess}, and Lemma \ref{lowerbound}, we have that
\begin{equation}\label{Fde}
\begin{aligned}
L_{i}(x_{i,k_1})-L_{i}(x_{i,k_1+1})&\geq \eta_1 (m_{i}(x_{i,k_1})-m_{i}(x_{i,k_1+1}))\\
&\geq \eta_1\kappa_{\rm fcd } h_{i,k_1}\min\left\{\frac{h_{i,k_1}}{1+\|B_{i,k_1}\|},\Delta_{i,k_1}\right\}\\
&\geq\eta_1\kappa_{\rm fcd }\epsilon_{\min}^h\min\left\{\frac{\epsilon_{\min}^h}{\kappa_H},\Delta_{\min}\right\},
\end{aligned}
\end{equation}
where $\epsilon_{\min}^h=\min_{i\in\{0,\,\dots,\,r\}}\epsilon_{i}^h$.
From Corollary \ref{eachonesuccess}, we know that there exists at least one successful iteration for every minimization sequence. Summing the objective decrease at level $i$ we obtain from \eqref{Fde} that, for any iteration $(i,\ell+1)$,
\begin{equation}\label{taubound}
    L_i(x_{i,0})-L_i(x_{i,\ell+1})=\sum_{j\in\{0,\,\dots,\,\ell\}\cap K_{s}}\left[L_i(x_{i,j})-L_i(x_{i,j+1})\right]\geq \tau_{i,\ell}\eta_1\kappa_h,
\end{equation}
where $\tau_{i,\ell}$ denotes the total number of successful iterations in level $i$ until iteration $\ell$, and $\kappa_h\coloneqq\kappa_{\rm fcd }\epsilon_{\min}^h\min\left\{{\epsilon_{\min}^h}{\kappa_H^{-1}},\Delta_{\min}\right\}\in(0,1)$.
From the construction of the lower-level objective function (as in \Cref{sec:rmta}) and since $f_i$, $\phi_i$, and $F_i=f_i+\phi_i$ satisfy Assumption \ref{passump} for each $i$, $F_i$ is bounded below. Then,  \eqref{taubound} implies that $\tau_{i,\ell}$ must be finite. On the other hand, Lemma \ref{lowerbound} implies that the minimization sequence is finite, as otherwise $\Delta_{i,k}$ will converge to zero, which is impossible by Lemma \ref{lowerbound}. Thus, the total number of iterations at each level is finite.
\end{proof}
The global convergence property is a direct consequence of Theorem \ref{finiteite}.
\begin{corollary}
 Assume that Algorithm RMNTR is called at the uppermost level with $\epsilon_r^h=0$. Then 
\begin{equation}\label{gc}
   \liminf_{k\to\infty}h_{r,k}=0.  
\end{equation}  
\end{corollary}
\begin{proof}
    The proof is identical to that of \cite[Corollary 4.11]{gst}, using Theorem~\ref{finiteite} and replacing $\|g_{i,k}\|$ by $h_{i,k}$.
\end{proof}

\section{Numerical results}\label{sec:numerics}
We now report on the practical performance of our method on two problem classes, PDE-constrained optimization and scientific machine learning, respectively.
All experiments were conducted on a MacBook Pro equipped with an Apple M2 Pro processor and 16~GB of RAM.
The implementation, together with all scripts necessary to reproduce the numerical results, is openly available at the public GitHub repository\footnote{\url{https://github.com/qiwang7777/Multilevel_LM.git}}.

We use the following algorithmic parameters for all the examples: $\Delta_0=50$, $\eta_1 = 0.05$, $\eta_2 = 0.95$, $\gamma_1=\gamma_2=0.25$, $\gamma_3=2$, $\kappa_{\rm stop}=0.6$, and $\epsilon_{i-1}=0.1$ when $i>1$. We stop Algorithm \ref{RMNTR} if $h_{r,k}\leq 10^{-7}$.
\subsection{Optimal control of Burger's equation}

Our first example is the optimal control of Burger's equation 
similar to \cite{bk2}:
\begin{equation}\label{opB}
    \min_{z\in L^2(\Omega)}\frac{1}{2}\int_{\Omega}([S(z)](x)-u_d(x))^2\,dx+\frac{\alpha}{2}\int_{\Omega}z(x)^2\,dx+\beta\int_{\Omega}|z(x)|\,dx
\end{equation}
where $\Omega=(0,1)$ is the physical domain, $\alpha=10^{-4}$ and $\beta=10^{-2}$ are fixed parameters (control costs). We choose $u_d(x)=-x^2$ as the target state. To emulate noisy measurement data, we perturb the exact target state 
\(u_d(x) = -x^2\) by adding a structured noise field consisting of three
components: (i) a dominant piecewise-constant ``step'' noise with several
randomly located jumps (maximum amplitude \(5\times 10^{-2}\)); 
(ii) an additional block noise generated from randomly placed constant 
intervals of comparable magnitude; and (iii) a sparse salt--and--pepper 
component with density \(0.5\%\) and spike amplitude \(0.2\). 
The noise is applied only at interior grid nodes so that the Dirichlet 
boundary values remain unperturbed.
Furthermore, we set $S(z)=u\in H^1(\Omega)$, where $u$ solves the weak form of Burgers' equation
\begin{equation}\label{B1}
    \begin{aligned}
    -\nu u''+uu'=z+g\quad\text{in $\Omega$},\\
    u(0)=0,\quad u(1)=-1,
    \end{aligned}
\end{equation}
with $g(x)=2(\nu+x^3)$ and $\nu=0.08$. We discretize the state $u$ using globally continuous piecewise linear finite elements and the control $z$ using piecewise constants on a uniform mesh with $n$ sub-intervals. 
Here, we simply choose the 1D averaging restriction operator $R_i^{i-1}$ as follows:
$$R_{i}^{i-1}=\frac{1}{\sqrt{2}}\begin{bmatrix}
1 & 1 & \quad &\quad&\quad\\
\quad & \quad & 1 & 1&\quad\\
\quad &\quad &\quad &\quad &\ddots
\end{bmatrix}\in \rr^{n_{i-1}\times n_{i}},$$
where $n_r=n.$
\begin{figure}[!htb]
    \centering
    \includegraphics[trim={7.5cm 0 7.5cm 0},angle=270,width=\linewidth]{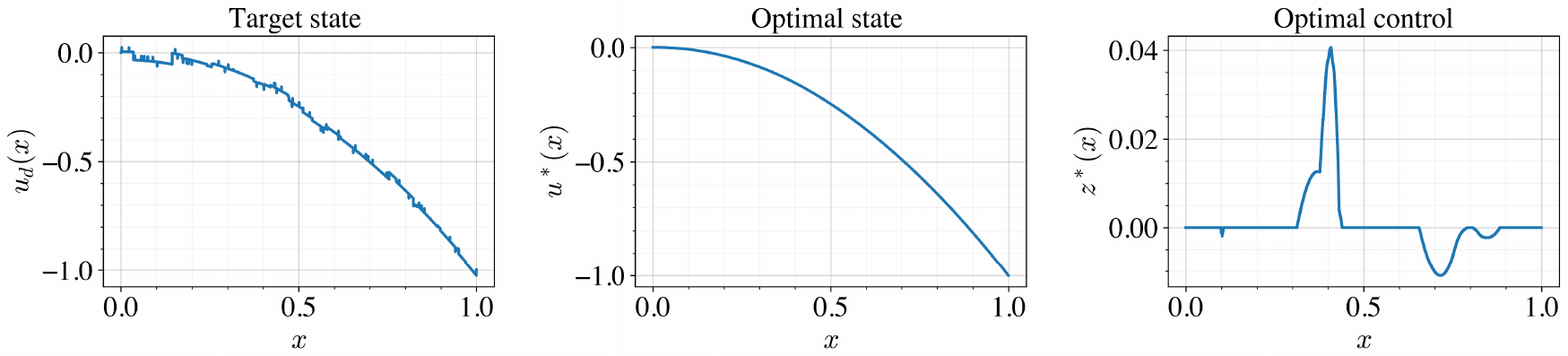}
    \caption{Optimal Control for Burgers with noise, target state with step (0.05), block (0.005), and impulse (0.2) noise}
    \label{fig:Optimal Control for Burgers}
\end{figure}

\begin{figure}[!htb]
    \centering
    \includegraphics[trim={6cm 0 6cm 0},angle=270,width=\linewidth]{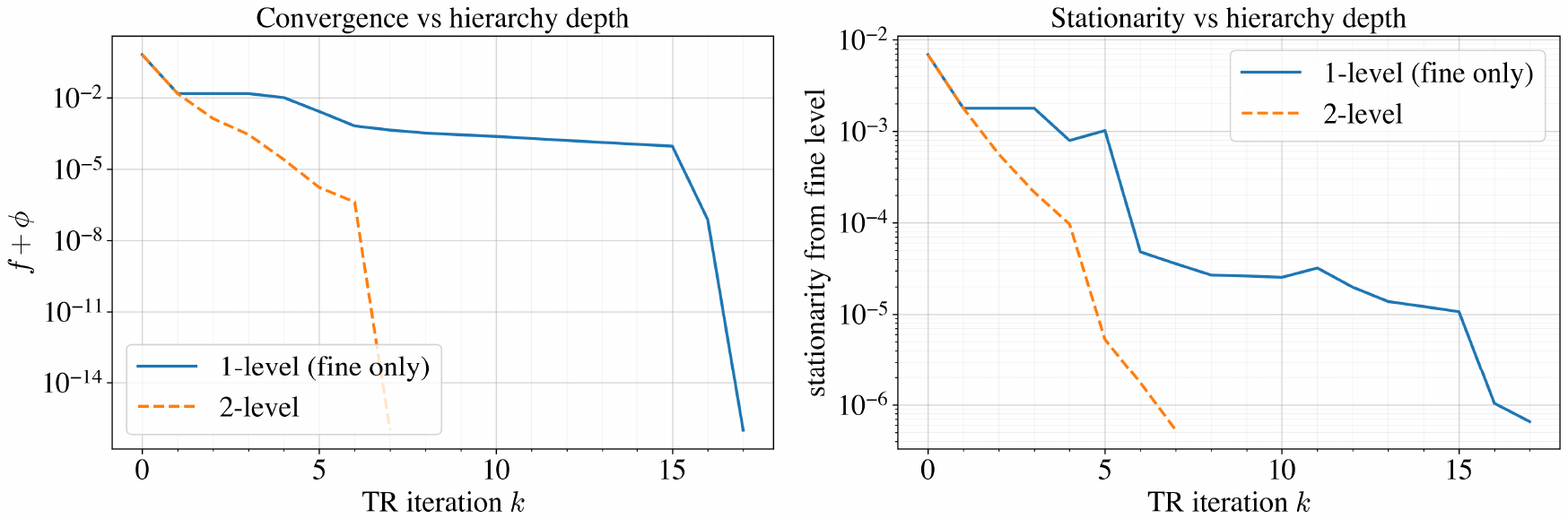}
    \caption{Convergence comparison with $n=8192$}
    \label{fig:com_burgers}
\end{figure}

Figure \ref{fig:Optimal Control for Burgers} shows the noisy target state $u_d$, the optimal state $u$, and the optimal control $z$. The target data is corrupted and nonsmooth, which creates local irregularities that the control must balance against the PDE and $L^2$- and $L^1$-regularization. The central plot of \Cref{fig:Optimal Control for Burgers} depicts an optimal state that is much smoother than the noisy target. The optimal control is relatively sparse, with localized spikes and oscillations. 

The left plot of Figure \ref{fig:com_burgers} compares the convergence behavior of the 1-level and the 2-level methods. It is obvious that the 2-level method converges significantly faster, which implies that coarse levels provide inexpensive but globally informative corrections that reduce the number of fine-level iterations. The right plot of Figure \ref{fig:com_burgers} compares the stationarity measure (i.e., a first‐order optimality indicator, $h_{r,k}$) for the 1-level and 2-level strategies. 
The 2-level run reduces the stationarity residual quickly, reaching very small $h_{k,r}$ after only a few iterations.
In contrast, the 1-level method reaches similar values after double the number of iterations and with more oscillation. 
Thus, the 2-level method not only reduces the objective faster but also drives the solution toward optimality conditions more rapidly and with smoother progress.

\subsection{Semilinear optimal control}
Our second example is the optimal control of a semilinear elliptic PDE, given by
\begin{equation}
\begin{aligned}
    \min_{z\in L^2(\Omega)}&\;\:\frac{1}{2}\int_{\Omega}([S(z)](x)-w(x))^2\,dx+\frac{\alpha}{2}\int_{\Omega}z(x)^2\,dx+\beta \int_{\Omega}|z(x)|\,dx\\
    \text{subject to} & \;\: -25\leq z\leq 25\quad\text{almost everywhere (a.e.),}
    \end{aligned}
\end{equation}
where $\Omega=(0,1)^2$ is the physical domain, and for the $L^2$-control cost we choose $\alpha=10^{-4}$. Concerning the $L^1$-cost we test $\beta = 0.05$ as well as $\beta=0.01$. We note that it is well known that the larger $\beta$ gets, the smaller becomes the support set for the optimal control. Further, $w\equiv-1$ is the target state, and $u= S(z)\in H^1(\Omega)$ solves the weak form of the semilinear elliptic PDE
\begin{equation}\label{semilinearpde}
-\Delta u +u^3 = z \quad\text{in $\Omega$}, \quad u=0 \quad\text{on $\partial\Omega$}.
\end{equation}
We discretize the state $u$ using piecewise linear finite elements on a uniform triangular mesh with $2n^2$ elements and the control variable $z$ using piecewise constants on the same mesh, resulting in $2n^2$ degrees of freedom. 
Here, we apply the restriction operator $R_{i}^{i-1}=R_{2D}=R_{1D}\otimes R_{1D}$, where $R_{1D}$ performs local averaging over adjacent grid points in one dimension. Specifically, each coarse-grid point corresponds to the average of 
$m$ consecutive fine-grid points, i.e.,
\[
R_{1D}[i,j] = 
\begin{cases}
\frac{1}{\sqrt{m}}, & \text{if } j \in \{i \cdot m,\, i \cdot m + 1,\, \dots,\, (i+1) \cdot m - 1\}, \\
0, & \text{otherwise},
\end{cases}
\]
for $i = 0, \dots, n_{i-1}-1$ and and $m=\frac{n_{i-1}}{n_i}$ being the coarsening ratio.

\begin{figure}[!htb]
    \centering
    \includegraphics[trim={6cm 0 6cm 0},angle=270,width=\linewidth]{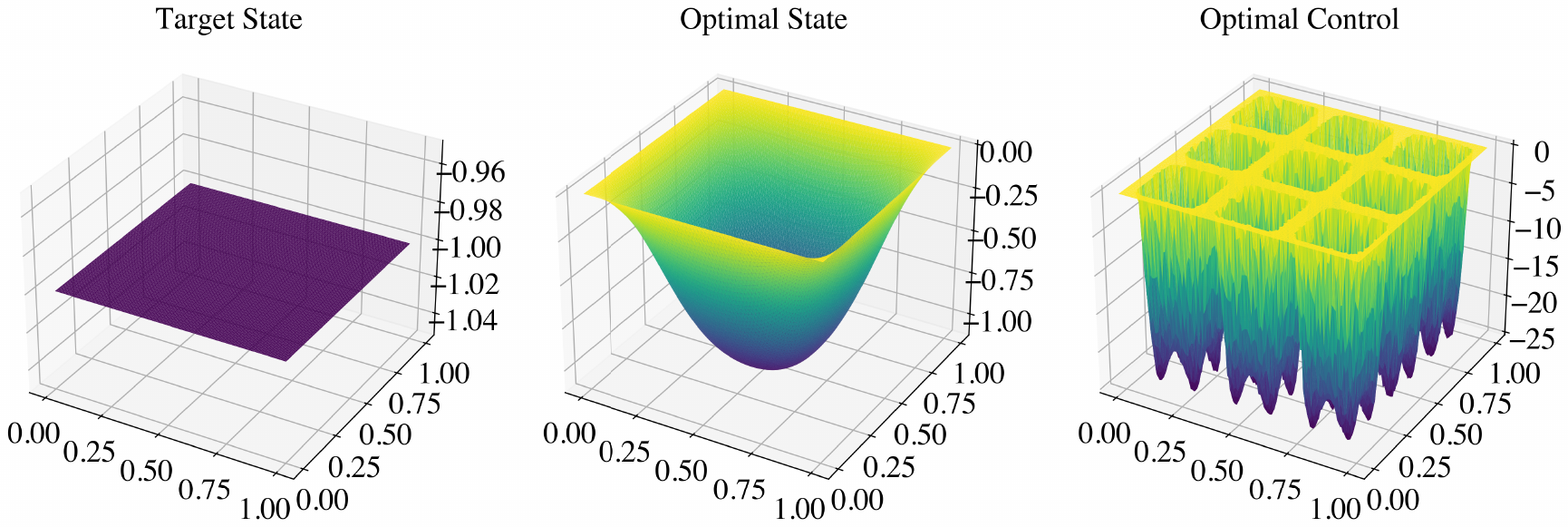}
    \caption{Optimal control for semilinear PDE with $n=128$, $\alpha=10^{-4}$ and $\beta=0.01$}
    \label{s1}
\end{figure}

\begin{figure}[!htb]
    \centering
    \includegraphics[trim={6cm 0 6cm 0},angle=270,width=\linewidth]{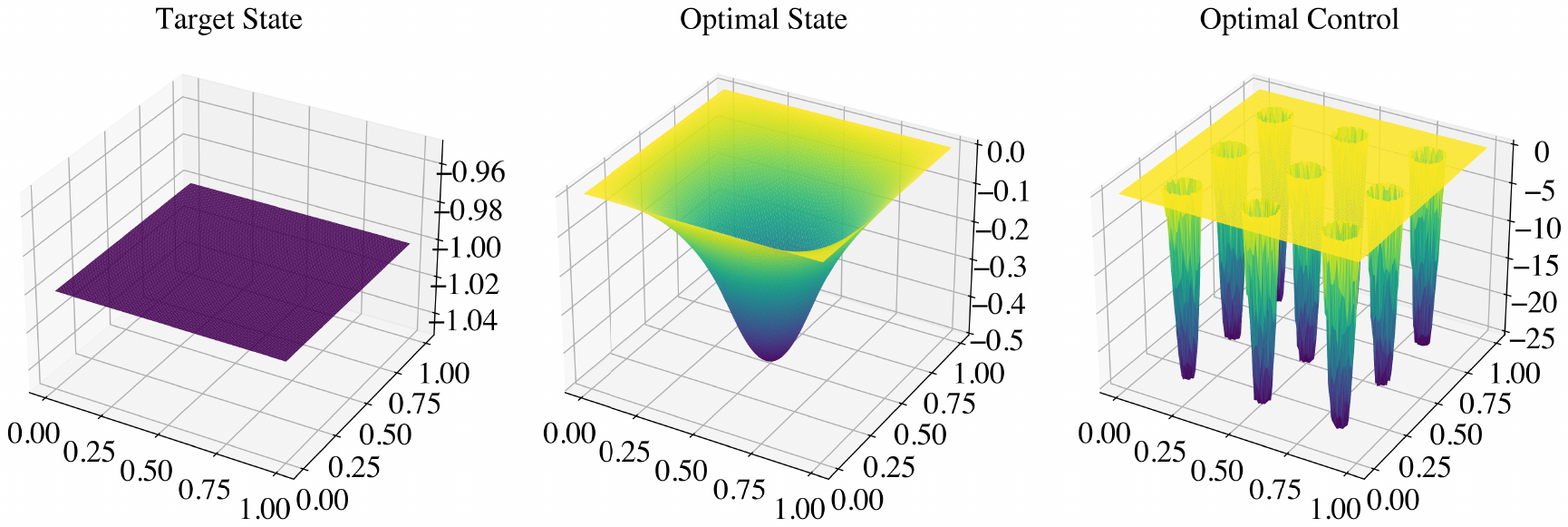}
    \caption{Optimal control for semilinear PDE with $n=128$, $\alpha=10^{-4}$ and $\beta=0.05$}
    \label{s2}
\end{figure}
We also consider the target function $w=w^{*}+\delta$, where $w^{*}\equiv -1$ and $\delta$ is a (discrete) realization of zero-mean Gaussian noise of standard deviation $\widehat{\sigma}$. We conducted experiments for various values of $\widehat{\sigma}$. However, since the numerical results are qualitatively similar across all tested noise levels, we report here only the case $\widehat{\sigma}=0.5$.
\begin{figure}[!htb]
    \centering
    \includegraphics[trim={6cm 0 6cm 0},angle=270,width=\linewidth]{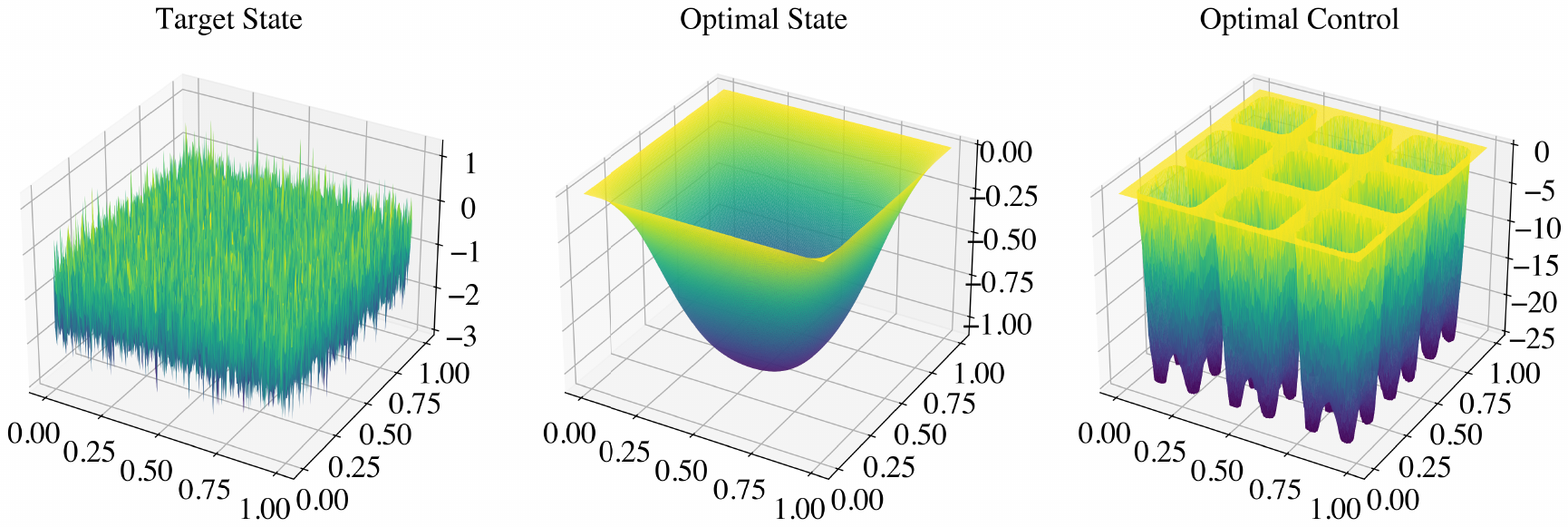}
    \caption{Optimal control for semilinear PDE with $n=256$, $\alpha=10^{-4}$, $\beta=0.01$, and $\widehat{\sigma}=0.5$}
    \label{snoise}
\end{figure}
\begin{figure}[!htb]
    \centering
    \includegraphics[trim={6cm 0 6cm 0},angle=270,width=\linewidth]{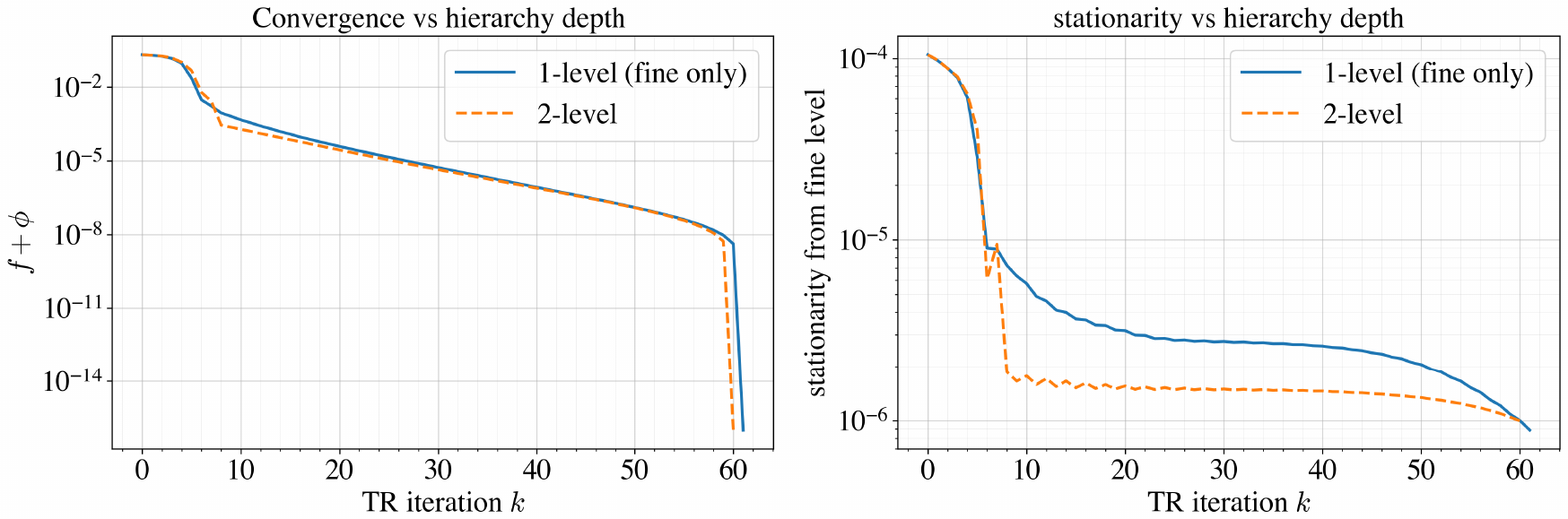}
    \caption{Convergence comparison with $n=256$}
    \label{fig:seminilinear_compare}
\end{figure}

Figures \ref{s1}, \ref{s2}, and \ref{snoise} all illustrate optimal control for the semilinear elliptic PDE constraint under different parameter settings and noise levels. As of course expected, the PDE constraint smooths the state, producing physically meaningful optimal states regardless of the noise level. Further, the $L^1$-regularization creates structured, sparse controls, which remain stable even with severe noise as in Figure \ref{snoise}. Comparing Figure \ref{s1} and \ref{s2}, we see that if we increase $\beta$, then it will sparsify the optimal control.

Figure \ref{fig:seminilinear_compare} presents the performance comparison between the 1-level method and the 2-level method for solving the semilinear PDE-constrained optimal control problem with a fine grid size $n=256$. The left subplot shows that the 2-level method accelerates convergence and reduces the computational effort (since the iteration count is approximately the same, but the 2-level method involves lower dimensional subproblems). The right subplot shows that the 2-level method reduces the stationarity residual much faster when compared to the 1-level method. 
\subsection{Neural network training}
For some given input $\omega\in\mathbb{R}^{n_0}$, a (artificial) neural network (NN) with $L$ hidden layers can be written as
$$\mathcal{N}(\omega)=(T_{L+1}\circ\sigma\circ T_{L}\circ\cdots\circ\sigma\circ T_2\circ\sigma\circ T_1)(\omega),$$
where, for every $l= 1,...,L+1$, the hidden layers consist of weight matrices $W_l\in\mathbb{R}^{n_l\times n_{l-1}}$ and bias vectors $b_l\in\mathbb{R}^{n_l}$, and an activation
function $\sigma:\mathbb{R}\to\mathbb{R}$, which is applied componentwise to vectors. Further, $T_l$ denotes the affine transformation $z\rightarrow W_lz +b_l$ from layer $l-1$ to layer $l$.
Rather than discretizing a PDE solution by a finite element ansatz as in the previous examples, we now parametrize the unknown solution $u$ by a neural network and aim to learn the parameters $\theta:=(W_{L+1}, b_{L+1},\cdots,W_1,b_1)$ from measurements respectively evaluations of the PDE data at sampling (i.e., collocation) points $\{\omega_i\}_{i=1}^M$. Below we also write $\mathcal{N}(\theta,\omega)$ in order to emphasize the dependence of $\mathcal{N}$ on $\theta$.This approach has become popular recently in the context of so-called PINNs, {\it physics-informed neural networks}, and is in particular interesting when, e.g., the geometry of the underlying PDE-domain is complicated, which challenges finite element discretizations, or the dimension (of the PDE-domain or the image space of $u$) is high. Here, we primarily focus on validating our theory and, hence, leave such challenging cases for future research.

Specifically, our goal here is to apply our multilevel solver to a PINN example with an $L^1$-penalty in the associated objective in order to favor sparsity in $\theta$ and, hence, of the NN. We note that the lower level problems in our solver are then connected to coarsenings of the underlying NN. For this purpose we consider the following second-order linear elliptic PDE in divergence form:
\begin{equation}\label{poisson}
    -\nabla\cdot(\kappa\nabla u)=g\quad\text{in}\quad \Omega:=(0,1)^2,\quad \text{and }u=0\text{ on }\partial\Omega,
\end{equation}
where $\partial \Omega$ denotes the boundary of $\overline{\Omega} (=\Omega^\circ \cup\partial\Omega)$ and the coefficient function is
\begin{equation*}
\kappa(x,y) = 1.1+0.2\sin(2\pi x)\cos(2\pi y),
\end{equation*}
for $(x,y)^\top\in\Omega$. We choose the exact reference solution to be
$$u^{*}(x,y) = x(1-x)y(1-y)[1+0.25\sin(2\pi x)\sin(2\pi y)+0.1xy].$$
Note that this choice and the form of the underlying PDE imply a specific $g$. 

In our multilevel PINN solver, we then consider $\kappa$ and $g$ as given and seek to compute an approximation of $u^*$ by solving the minimization problem
\begin{equation}
    \min_{\theta\in\rn}\frac{1}{2|\Omega^{\circ}|}\sum_{\omega\in\Omega^{\circ}}\|-\nabla\cdot(\kappa\nabla \mathcal{N})(\theta,\omega)-g(x)\|_{L^2}^2+\frac{1}{2|\partial\Omega|}\sum_{\theta\in\partial\Omega}\|\mathcal{N}(\theta,\omega)\|_{L^2}^2+\beta\|\theta\|_{L^1}
\end{equation}
where $\theta$ is the collection of NN parameters, and $\beta>0$ is a given weight. In our tests, we consider a uniform $32\times 32$ grid in $\overline{\Omega}$, where $|\Omega^\circ|$ is the number of interior grid points and $|\partial\Omega|$ is the number of boundary grid points. These grid points also represent the collection of collocation points $\{\omega_i\}_{i=1}^M$, with $M=32^2$.
Our chosen neural network architecture is a fully connected one with one hidden layer consisting of 60 neurons and with a standard sigmoidal activation function.

\begin{figure}[H]
    \centering
    \includegraphics[trim={6cm 0 6cm 0},angle=270,width=\linewidth]{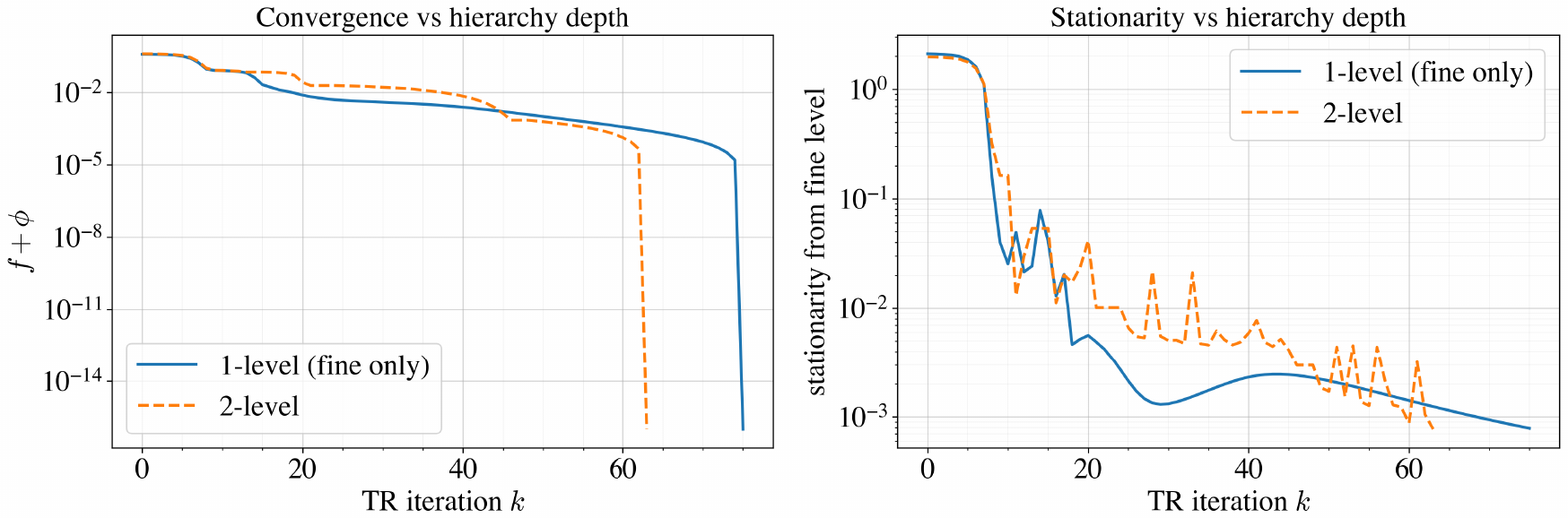}
    \caption{Convergence comparison between different level for PINNs example with $n=240$. }
    \label{fig:pinn_con_com}
\end{figure}

Figure \ref{fig:pinn_con_com} compares the convergence behavior of the 1-level method and the 2-level method when training a PINN for \eqref{poisson}. It shows that the 2-level method converges dramatically faster in terms of the loss function. It also reaches near-stationarity earlier though with more oscillation due to aggressive coarse-grid corrections.

\subsection{Performance summary}
In Table \ref{tab:performance} we summarize the performance in the tests reported above. We record the number of trust-region iterations (\texttt{iter}), the number of $f$ (\texttt{fval}) and $\nabla f$ (\texttt{grad}) evaluations, the number of $\nabla^2 f$ applications (\texttt{hess}), the number of $\phi$ evaluations (\texttt{phi}), the number of proximity operator evaluations (\texttt{prox}), and the wallclock time in seconds (\texttt{time} (s)).
\begin{table}[H]
\centering
\caption{Algorithmic performance (for all examples). 
Here,  \texttt{iter} = the number of trust-region iterations, 
\texttt{f/grad/hess} are the numbers of $f, \nabla f, \nabla^2 f$ evaluations, 
\texttt{phi/prox} are the numbers of $\phi$, proximity-operator evaluations, 
and \texttt{time} is the wall-clock time in seconds.}
\label{tab:performance}
\begin{tabular}{llcccccccc}
\toprule
Example & DoF & levels & {\tt iter} & {\tt fval} & {\tt grad} & {\tt hess} & {\tt phi} & {\tt prox} & {\tt time} \\
\midrule
Burgers &8,192& 1 & 24 & 25 & 21 & 550 & 379   & 310  & 11.6 \\
Burgers &8,192& 2 & 4 & 7 & 6 & 218 & 283  & 216  & 9.1 \\
Burgers & 8,192&3 & 4 & 7 & 6 & 250 & 339   & 286   & 10.5 \\
Semilinear &32,768  & 1 & 8 & 9 & 9 & 136 & 109 & 80 & 21.7 \\
Semilinear&32,768 & 2 & 7 & 12 & 10 & 98 & 87   & 78   & 18.6 \\
Semilinear& 131,072 & 1 & 61 & 62 & 62 & 366 & 393 & 301 & 204.6 \\
Semilinear& 131,072 & 2 & 60 & 71 & 66 & 347 & 402 & 368 & 158.4 \\
NNs & 240 & 1 & 372 & 373 & 366 & 13756 & 19982  &14235 & 122.9\\    
NNs & 240 & 2 & 168& 187  &  165  & 10139 & 14991 &  10541 & 92.2\\
\bottomrule
\end{tabular}
\end{table}
Table \ref{tab:performance} reports computional efficiency for different examples. Across all examples, using the multilevel method leads to shorter runtime compared to single level, which implies that the multilevel strategy accelerates convergence significantly and consistently. Moreover, compared to the single level iteration, the multilevel method requires fewer trust-region iterations. Also, the numbers of gradient and Hessian evaluations drop sharply with more levels. Therefore, multilevel strategies reduce costly second-order and proximal operations, improving efficiency.
\section{Conclusion}
In this work, we developed a recursive trust-region algorithm for minimizing the sum of a smooth nonconvex function and a nonsmooth convex function in $\rn$. Our algorithm employs the proximal gradient step as a generalization of the Cauchy point, which allows us to prove the existence of a trial step and ultimately global convergence of the trust-region algorithm. The numerical experiments confirm that the proposed multilevel proximal trust-region (RMNTR) method is highly efficient and robust across various nonsmooth optimization problems, including PDE-constrained optimal control and physics-informed neural network (PINN) training. In the Burgers’ equation control problem, the multilevel algorithm produced results of comparable accuracy while achieving a clear computational advantage over the single-level approach. For the semilinear elliptic optimal control problem, the method maintained reliable convergence on increasingly fine meshes, demonstrating strong scalability and reduced overall computational effort. In the PINN example, the RMNTR method required far fewer training iterations and significantly less runtime, while achieving the same loss accuracy and stationarity as the baseline. Overall, these results indicate that incorporating a multilevel hierarchy into the proximal trust-region framework brings substantial performance benefits, effectively accelerating convergence without compromising solution quality, and making it a powerful approach for large-scale nonsmooth optimization in scientific machine learning and PDE-constrained applications.

Future work may focus on exploring alternative proximal or quasi-Newton-based subsolvers within the multilevel trust-region framework, which could further improve efficiency and robustness, especially for large-scale nonsmooth problems. Moreover, extending the approach to more complex and higher-dimensional PINN architectures—such as deep or adaptive networks with intricate physical constraints—could provide additional insights into the scalability and flexibility of the RMNTR method in modern scientific machine learning contexts. For PINNs, we also highlight \cite{baraldi2025proxstorm} as a relevant direction for future research, particularly in exploring stochastic data sampling strategies within multilevel frameworks.

\section*{Acknowledgements}
MH and QW gratefully acknowledge support from the DFG Excellence Cluster MATH+ (EXC 2046) through the project AA5-2 ``Robust Multilevel Training of Artificial Neural Networks.'' MH also acknowledges the ``Ettore Majorana Foundation and Center for Scientific Culture'' in Erice, Sicily, where part of this research was performed.

\section*{Funding}

Funded by the Deutsche Forschungsgemeinschaft (DFG, German Research Foundation) under Germany's Excellence Strategy – The Berlin Mathematics Research Center MATH+ (EXC-2046/1, project ID: 390685689).

This article has been co-authored by an employee (RB) of National Technology \& Engineering
Solutions of Sandia, LLC under Contract No. DE-NA0003525 with the U.S. Department of Energy (DOE). 
The employee owns all right, title and interest in and to
the article and is solely responsible for its contents. The United States Government
retains and the publisher, by accepting the article for publication, acknowledges that
the United States Government retains a non-exclusive, paid-up, irrevocable, world-
wide license to publish or reproduce the published form of this article or allow others
to do so, for United States Government purposes. The DOE will provide public access
to these results of federally sponsored research in accordance with the DOE Public
Access Plan: \texttt{https://www.energy.gov/downloads/doe-public-access-plan}
This paper describes objective technical results and analysis. Any subjective views or
opinions that might be expressed in the paper do not necessarily represent the views
of the U.S. Department of Energy or the United States Government.
This work was supported by the Sandia Laboratory Directed Research and
Development Program. 

\bibliographystyle{siam}  
\bibliography{interacttfssample}      
\appendix

\section{Additional proofs for multiple levels} \label{appendix}

So far our proofs cover mainly the two-level case, only. 
Now, we generalize to multiple levels in the same manner as \cite{gst}. 
In fact, many of the proofs follow quite closely, with the necessary substitutions to accommodate the nonsmooth term $\phi$. 
Here, we follow the notations from~\cite{gst} and refer the reader to that reference for further details.
\begin{itemize}
    \item [1.] If iteration $(i,k)$ is recursive, then this iteration initiates a \emph{minimization sequence} at level $i-1$.
    This sequence consists of all successive iterations at this level starting from $x_{i-1,0}=R_i^{i-1} x_{i, k}$
    until a return is made to level $i$ within iteration $(i, k)$. 
    If such a return occurs, then the iteration $(i, k)$ is the \emph{predecessor} of level $i-1$'s minimization sequence. 
    We write $(i, k)=\pi(i-1, \ell)$ to signify that $(i-1, \ell)$ belongs to this minimization sequence.
    \item [2.]  For some iteration $(i, k)$, the set
$$
\mathcal{R}(i, k) \coloneqq\{(j, \ell) \mid \text { iteration }(j, \ell) \text { occurs within iteration }(i, k)\}
$$
to be the iterations within each sequence.
Additionally, we have $j \leq i$ for every $j$ such that $(j, \ell) \in \mathcal{R}(i, k)$ for some nonnegative $k$ and $\ell$
and
\begin{equation}\label{Deltamono}
\Delta_{j, \ell} \leq \Delta_{i, k}, \text { whenever }(j, \ell) \in \mathcal{R}(i, k),    
\end{equation}
because of the choice of $\Delta_{j, 0}$ in Step 0 and \eqref{deres}. 
\item [3.] 
We define the index of the deepest level within recursion $(i,k)$ as
$$
\zeta(i, k)\coloneqq\min _{(j, \ell) \in \mathcal{R}(i, k)} j.
$$ 
The path from $(\zeta(i, k), \ell)$ to $(i, k)$ is the longest in $\mathcal{R}(i, k)$.
\item [4.]  We write
$$
\mathcal{T}(i, k)\coloneqq\{(j, \ell) \in \mathcal{R}(i, k) \mid \text { iteration }(j, \ell) \text { is a Taylor iteration}\}
$$
for the subset of iterations within $\mathcal{R}(i, k)$ at which Taylor's model $m_{j, \ell}\left(x_{j, \ell}+s_j\right)$ is chosen.
\end{itemize}
We also define the constants
\begin{equation}\label{kpr}
    \kappa_{PR}\coloneqq\max\{1,\max_{i=1,\,\dots,\,r}\|P_{i-1}^{i}\|\}= \max\{1,\max_{i=1,\,\dots,\,r}\|R_{i-1}^{i}\|\},
\end{equation}
where we use the assumption $\sigma_i=1$ and
\begin{equation}\label{ksigma}
    \kappa_{\sigma}\coloneqq\min\{1,\min_{i=0,\,\dots,\,r}\sigma_{\min}(M_i)\}>0,
\end{equation}
where $\sigma_{\min}(A)$ denotes the smallest singular value of the matrix $A$. We finally define 
\begin{equation}\label{46}
\Delta_{\min}^s\coloneqq\min_{i=0,\dots,r}\Delta_i^s,\quad\epsilon_{\min}^h\coloneqq\min_{i=0,\dots,r}\epsilon_i^h,\quad\text{and}\quad\epsilon_{\min}^{\Delta}\coloneqq\min_{i=0,\dots,r}\epsilon_i^{\Delta}.
\end{equation}

First, we obtain the same results as in \cite[Lemmas 4.1 and 4.2]{gst} by replacing $\|g_{j,\ell}\|$ and $\epsilon_j^g$ with $h_{j,\ell}$ and $\epsilon_j^h$ respectively; these results ensure that the steps remain within the trust region. We refer the reader to the statements and proofs of \cite[Lemmas 4.1 and 4.2]{gst}, as the presence of $\phi$ does not affect the argument.

We now prove some useful bounds on the proximal gradient norms for all iterates that belong to a recursion process initiated within a sufficiently small trust region.
A similar result appears in \cite[Lemma 4.3]{gst}. 
\begin{lemma}
    For some iteration $(i,k)$ assume that
    \begin{equation}\label{deltasmall}
        \Delta_{i,k}\leq \frac{\sqrt{\kappa_{\sigma}}\kappa_{\rm stop}^r}{2r(\kappa_H+2/t)}\eqqcolon\kappa_1 h_{i,k},
    \end{equation}
    where $\kappa_1\in(0,1)$. Then one has for all $(j,\ell)\in\mathcal{R}(i,k)$ that
    \begin{equation}\label{equh}
        \frac{1}{2}\kappa_{\rm stop}^r h_{i,k}\leq h_{j,\ell}\leq\kappa_{PR}^r(1+\frac{1}{2}\kappa_{\rm stop}^r)h_{i,k}.
    \end{equation}
\end{lemma}
\begin{proof}
    The result is obvious for $(j,\ell)=(i,k)$.
    Consider now some iteration $(j,\ell)\in\mathcal{R}(i,k)$ with $j<i$. The definition of the proximal gradient yields
    for $(j,\ell)$ that
\begin{equation}
\begin{aligned}
    &h_{j,\ell}
    = \frac{1}{t}\left\|
         x_{j,\ell}
         - \Prox_{t\phi_j}\!\left(
             x_{j,\ell} - t\nabla \tilde{L}_j(x_{j,\ell})
           \right)
       \right\|_j \\[4pt]
    &\ge
      h_{j,0}
    - \frac{1}{t}\Bigg\|
          x_{j,\ell}-x_{j,0}  
          - \Prox_{t\phi_j}\!\left(
                x_{j,\ell}-t\nabla \tilde{L}_j(x_{j,\ell})
            \right) 
          + \Prox_{t\phi_j}\!\left(
                x_{j,0}-t\nabla \tilde{L}_j(x_{j,0})
            \right)
        \Bigg\|_j \\[4pt]
    &\eqqcolon h_{j,0} - \vartheta_{j,\ell}.
\end{aligned}
\end{equation}

    By the non-expansiveness of 'Prox' \cite[Part 3 of Lemma 1]{bk} and \eqref{boundhess}, we have
    \begin{equation}\label{Pupper}
    \begin{aligned}
        \vartheta_{j,\ell}&\leq \frac{2}{t}\|x_{j,\ell}-x_{j,0}\|_j+\|\nabla \tilde{L}_j(x_{j,\ell})-\nabla \tilde{L}_j (x_{j,0})\|_j\\
        &\leq \frac{2}{t}\|x_{j,\ell}-x_{j,0}\|_j+\kappa_H\|x_{j,\ell}-x_{j,0}\|_j=(2t^{-1}+\kappa_H)\|x_{j,\ell}-x_{j,0}\|_j.
    \end{aligned}
    \end{equation}
    For the last inequality above we use that if we choose the low-level model on level $j+1$, then we have $\nabla \tilde{L}_j = R_{j+1}^{j}\nabla \tilde{L}_{j+1}= R_{j+1}^{j}\nabla f_{j+1}$, otherwise $\nabla \tilde{L}_j = \nabla f_{j}$. We get \eqref{Pupper} for both cases by \eqref{boundhess} and Assumption \ref{roworthonormal}. Hence, by Remark \ref{reinorm} it holds that
    \begin{equation}\label{hjl0}
        h_{j,\ell}\geq h_{j,0}-\left(\frac{2}{t\sqrt{\kappa_{\sigma}}}+\frac{\kappa_H}{\sqrt{\kappa_{\sigma}}}\right)\|x_{j,\ell}-x_{j,0}\|_j
    \end{equation}
    for all $(j,\ell)$. On the other hand, if $(j+1,q)\in\pi(j,\ell)$, we also have that, for all $(j,\ell)\in\mathcal{R}(i,k)$, $(j+1,q)\in\mathcal{R}(i,k)$ and 
    \begin{equation}\label{xdelta}
      \|x_{j,\ell}-x_{j,0}\|_j\leq\Delta_{j+1,q}\leq\Delta_{i,k}  
    \end{equation}
    because of \eqref{Deltamono} and Lemma \cite[Lemma 4.1]{gst}. Combining \eqref{hjl0} and \eqref{xdelta}, we get that, for all $(j,\ell)\in\mathcal{R}(i,k)$,
    \begin{equation}
        h_{j,\ell}\geq h_{j,0} - \left(\frac{2}{t\sqrt{\kappa_{\sigma}}}+\frac{\kappa_H}{\sqrt{\kappa_{\sigma}}}\right)\Delta_{i,k}.
    \end{equation}
    Consider now the path from $(j,\ell)$ to $(i,k)$ in $\mathcal{R}(i,k)$. Let this path consist of the iterations $(j,\ell),(j+u,\mu_{j+u})$ for $u=1,\,\dots,\,i-j-1$, and $(i,k)$. First we note that,
     Definition \ref{lowhigh}, Assumption \ref{roworthonormal}, and ~\cite[Theorem 6.15]{B} ensure that
\begin{equation*}
\begin{aligned}
\Prox_{t\phi_{j}}\!\Bigl(
    x_{j,0}
    - t R_{j+1}^{j}
      \nabla \tilde{L}_{j+1}(x_{j+1,\mu_{j+1}})&
\Bigr) \\=
R_{j+1}^{j}&\,
\Prox_{t\phi_{j+1}}\!\Bigl(
    x_{j+1,\mu_{j+1}}
    - t\nabla \tilde{L}_{j+1}(x_{j+1,\mu_{j+1}})
\Bigr).
\end{aligned}
\end{equation*}

     We then have that 
   
        \begin{align}
           h_{j,0}&=\frac{1}{t}\|x_{j,0}-\Prox_{t\phi_j}(x_{j,0}-t\nabla \tilde{L}_j(x_{j,0}))\|_j\notag\\
           &=\frac{1}{t}\|x_{j,0}-\Prox_{t\phi_j}(x_{j,0}-t R_{j+1}^{j}\nabla \tilde{L}_{j+1}(x_{j+1,\mu_{j+1}}))\|_j\notag\\
           &=\frac{1}{t}\|R_{j+1}^{j}(x_{j+1,\mu_{j+1}}-\Prox_{t\phi_{j+1}}(x_{j+1,\mu_{j+1}}-t\nabla \tilde{L}_{j+1}(x_{j+1,\mu_{j+1}})))\|_j\label{hj0r}\\
           &\geq \kappa_{\rm stop} h_{j+1,\mu_{j+1}}.\notag
        \end{align}
    
    Thus, combining all the inequalities above and \eqref{fail} with $\kappa_{\rm stop}\in(0,1)$, we obtain
    \begin{equation}\label{hj0u}
        \begin{aligned}
          h_{j,\ell}&\geq h_{j,0}- \left(\frac{2}{t\sqrt{\kappa_{\sigma}}}+\frac{\kappa_H}{\sqrt{\kappa_{\sigma}}}\right)\Delta_{i,k}\\
          &\geq \kappa_{\rm stop} h_{j+1,\mu_{j+1}}-\left(\frac{2}{t\sqrt{\kappa_{\sigma}}}+\frac{\kappa_H}{\sqrt{\kappa_{\sigma}}}\right)\Delta_{i,k}\\
          &\geq \kappa_{\rm stop} h_{j+1,0}-2\left(\frac{2}{t\sqrt{\kappa_{\sigma}}}+\frac{\kappa_H}{\sqrt{\kappa_{\sigma}}}\right)\Delta_{i,k}\\
          &\geq\kappa_{\rm stop}^2 h_{j+2,\mu_{j+2}}-2\left(\frac{2}{t\sqrt{\kappa_{\sigma}}}+\frac{\kappa_H}{\sqrt{\kappa_{\sigma}}}\right)\Delta_{i,k}\\
          &\geq\kappa_{\rm stop}^r h_{i,k}-r\left(\frac{2}{t\sqrt{\kappa_{\sigma}}}+\frac{\kappa_H}{\sqrt{\kappa_{\sigma}}}\right)\Delta_{i,k}.
        \end{aligned}
    \end{equation}
    We then deduce the first inequality of \eqref{equh} from \eqref{deltasmall}.

    To prove the second, we first note that, for any iteration $(j,\ell)$,
    \begin{equation}\label{hj0l}
        \begin{aligned}
            h_{j,\ell} &= \frac{1}{t}\Big\|x_{j,0}-\Prox_{t\phi_j}(x_{j,0}-t\nabla \tilde{L}_j(x_{j,0}))+(x_{j,\ell}-x_{j,0})\\
            &\quad-(\Prox_{t\phi_j}(x_{j,\ell}-t\nabla \tilde{L}_j(x_{j,\ell}))-\Prox_{t\phi_j}(x_{j,0}-t\nabla \tilde{L}_j(x_{j,0})))\Big\|_j\\
            &\leq h_{j,0} + \vartheta_{j,\ell}\\
        &\leq h_{j,0} +\left(\frac{2}{t\sqrt{\kappa_{\sigma}}}+\frac{\kappa_H}{\sqrt{\kappa_{\sigma}}}\right)\|x_{j,\ell}-x_{j,0}\|_j\\
        &\leq h_{j,0} +\left(\frac{2}{t\sqrt{\kappa_{\sigma}}}+\frac{\kappa_H}{\sqrt{\kappa_{\sigma}}}\right)\Delta_{i,k}.
        \end{aligned}
    \end{equation}
    We now retrack the iteration path from $(j,\ell)$ back to $(i,k)$ as above and successively deduce from \eqref{hj0l}, \eqref{hj0r}, \eqref{fail}, and \eqref{kpr} that
    \begin{equation*}
        \begin{aligned}
            h_{j,\ell}&\leq h_{j,0}+\left(\frac{2}{t\sqrt{\kappa_{\sigma}}}+\frac{\kappa_H}{\sqrt{\kappa_{\sigma}}}\right)\Delta_{i,k}\\
            &\leq \kappa_{PR}h_{j+1,\mu_{j+1}}+\left(\frac{2}{t\sqrt{\kappa_{\sigma}}}+\frac{\kappa_H}{\sqrt{\kappa_{\sigma}}}\right)\Delta_{i,k}\\
            &\leq\kappa_{PR} h_{j+1,0}+ \left((\kappa_{PR}+1)\left(\frac{2}{t\sqrt{\kappa_{\sigma}}}+\frac{\kappa_H}{\sqrt{\kappa_{\sigma}}}\right)\right)\Delta_{i,k}\\
            &\leq\kappa_{PR}^2 h_{j+2,\mu_{j+2}}+2\kappa_{PR}\left(\frac{2}{t\sqrt{\kappa_{\sigma}}}+\frac{\kappa_H}{\sqrt{\kappa_{\sigma}}}\right)\Delta_{i,k}\\
            &\leq\kappa_{PR}^rh_{i,k}+r\kappa_{PR}\left(\frac{2}{t\sqrt{\kappa_{\sigma}}}+\frac{\kappa_H}{\sqrt{\kappa_{\sigma}}}\right)\Delta_{i,k}\\
            &\leq\kappa_{PR}^r\left[h_{i,k}+r\left(\frac{2}{t\sqrt{\kappa_{\sigma}}}+\frac{\kappa_H}{\sqrt{\kappa_{\sigma}}}\right)\Delta_{i,k}\right]
        \end{aligned}
    \end{equation*}
    using $\kappa_{PR}\geq 1$. We use the bound \eqref{deltasmall} to conclude the second inequality of \eqref{equh}.
\end{proof}
Following \cite{gst}, we investigate noncritical points with small trust region radius $\Delta_{i,k}$. 
Let $\delta_{j,\ell}\coloneqq \pred_{j,\ell}$ and consider $\mathcal{V}(i,k)\subset\mathcal{R}(i,k)$ defined by
\begin{equation}\label{Vik}
    \mathcal{V}(i,k)\coloneqq\left\{(j,\ell)\in\mathcal{R}(i,k)\,|\,\delta_{j,\ell}\geq \frac{1}{2}\kappa_{\rm fcd }\kappa_{\rm stop}^r\kappa_{\epsilon}^{j-\zeta(i,k)}h_{i,k}\Delta_{j,\ell}\right\},
\end{equation}
where 
\begin{equation}\label{427}
    \kappa_{\epsilon}\coloneqq\eta_2\epsilon_{\min}^{\Delta}<1.
\end{equation}
The set $\mathcal{V}(i,k)$ is the subset of iterations within the recursion at iteration $(i,k)$ for which the model decrease is bounded below by a level-dependent factor times the product of the proximal gradient $h_{i,k}$ and the trust region radius $\Delta_{j,\ell}$. 
If iteration $(j,\ell)$ belongs to $\mathcal{V}(i,k)$, then $\delta_{j,\ell}$ can be computed in a finite number of iterations, implying $\mathcal{R}(j,\ell)$ is finite. 
We now proceed to show that $\mathcal{V}(i,k)$ and $\mathcal{R}(i,k)$ coincide for a sufficiently small radius $\Delta_{i,k}$, compare \cite[Theorem 4.4]{gst}. 
\begin{theorem}\label{t5}
Consider an iteration $(i,k)$ for which $h_{i,k}>0$ and
\begin{equation}\label{Deltaup}
    \Delta_{i,k}\leq\min[\Delta_{\min}^{s},\min\left\{\kappa_1,\frac{\kappa_{\rm fcd }\kappa_{\sigma}\kappa_{\rm stop}^r\kappa_{\epsilon}^r(1-\eta_2)}{2\kappa_H}\right\}h_{i,k}]\coloneqq\min[\Delta_{\min}^{s},\kappa_2 h_{i,k}],
\end{equation}
where $\kappa_2\in(0,1)$. Then the following conclusions hold:

\begin{itemize}
    \item [1.] every iteration using Taylor's model belongs to \eqref{Vik}, that is,
    \begin{equation}\label{TinV}
        \mathcal{T}(i,k)\subseteq\mathcal{V}(i,k), \quad\text{and}
    \end{equation}
    \item [2.] iteration $(j,\ell)$ is very successful for every $(j,\ell)\in\mathcal{V}(i,k).$
    Moreover, if all iterations $(j,\ell)$ of a minimization sequence at level $j<i$ belong to $\mathcal{V}(i,k)$ and if $\pi(j,\ell)=(j+1,q)$, then
    \item [3.] the decrease in the objective function at level $j$ satisfies, for each $\ell>0$,
    \begin{equation}\label{430}
        L_j(x_{j,0})-L_j(x_{j,\ell})\geq\frac{1}{2}\kappa_{\rm fcd }\kappa_{\rm stop}^r\kappa_{\epsilon}^{j-\zeta(i,k)+1}\ell h_{i,k}\Delta_{j+1,q},
    \end{equation}
    \item [4.] there are finite iterations
     in the minimization sequence at level $j$, and
    \item [5.] we have that
    \begin{equation}\label{jinV}
        (j+1,q)\in\mathcal{V}(i,k)
    \end{equation}
\end{itemize}    
\end{theorem}
\begin{proof}
We handle our proofs in the order presented, directing the reader to that of \cite[Theorem 4.4]{gst} 
where the logic is identical to ours. 
The proofs for items 1, 3, and 5 follow the same arguments as in \cite[Theorem 4.4]{gst}. By replacing $\|g_{i,k}\|$, $\kappa_{red}$, $\kappa_g$ with $h_{i,k}$,$\kappa_{\rm fcd }$, $\kappa_{\rm stop}$ respectively. The proofs for items 2 and 4 are slightly different, as our objective function $F$ need not be continuous, in general. But we could use Assumption \ref{passump} to overcome this difficulty. For the convenience of the reader, we provide the details for items 2 and 4 below.

We prove item 2 separately for $(j,\ell)\in\mathcal{T}(i,k)$ and for $(j,\ell)\in\mathcal{V}(i,k)\setminus\mathcal{T}(i,k)$. Consider the case where $(j,\ell)\in\mathcal{T}(i,k)$ first. For $(j,\ell)\in\mathcal{T}(i,k)$, by Taylor's theorem, condition \eqref{boundhess}, and $\|s_{j,\ell}\|_j\leq\Delta_{j,\ell}$, we have that
\begin{equation}\label{435}
    \left|L_j(x_{j,\ell}+s_{j,\ell})-m_{j,\ell}(x_{j,\ell}+s_{j,\ell})\right|\leq \kappa_H\left(\frac{\|s_{j,\ell}\|}{\|s_{j,\ell}\|_j}\right)^2\Delta_{j,\ell}^2.
\end{equation}
The definition of $\|\cdot\|_j$ and $\kappa_{\sigma}$ yield $\|s_{j,\ell}\|_{j}\geq\sqrt{\kappa_{\sigma}}\|s_{j,\ell}\|$. Therefore, \eqref{435} becomes
\begin{equation*}
    \left|L_j(x_{j,\ell}+s_{j,\ell})-m_{j,\ell}(x_{j,\ell}+s_{j,\ell})\right|\leq\frac{\kappa_H}{\kappa_{\sigma}}\Delta_{j,\ell}^2.
\end{equation*}

Combining this last bound with \cite[(4.34)]{gst}, we obtain that
\begin{equation*}
    |\rho_{j,\ell}-1|\leq\left|\frac{L_j(x_{j,\ell}+s_{j,\ell})-m_{j,\ell}(x_{j,\ell}+s_{j,\ell})}{m_{j,\ell}(x_{j,\ell})-m_{j,\ell}(x_{j,\ell}+s_{j,\ell})}\right|\leq\frac{2\kappa_H}{\kappa_{\rm fcd }\kappa_{\sigma}\kappa_{\rm stop}^rh_{i,k}}\Delta_{j,\ell}\leq 1-\eta_2,
\end{equation*}
where the last inequality is deduced from \eqref{Deltamono} and the fact that \eqref{Deltaup} implies
$
    \Delta_{i,k}\leq\kappa_{\rm fcd }\kappa_{\sigma}\kappa_{g}^rh_{i,k}(1-\eta_2)/2\kappa_{H},
$
since $\kappa_{\epsilon}<1$. Hence, $\rho_{j,\ell}\geq\eta_2$ and the iteration $(j,\ell)\in\mathcal{T}(i,k)$ is very successful, as requested in item 2. 

Next we prove item 2 for $(j,\ell)\in\mathcal{V}(i,k)\setminus\mathcal{T}(i,k)$, which implies, in particular, that $\mathcal{R}(j,\ell)$ is finite and $x_{j-1,*}$ is well defined. If we consider the iteration $(j,\ell)$, 
\begin{equation*}
\begin{aligned}
    &L_j(x_{j,\ell})-L_j(x_{j,\ell}+s_{j,\ell})\\
    &\quad=-\langle \nabla f_{j,\ell}(x_{j,\ell}),s_{j,\ell}\rangle-\frac{1}{2}\langle\nabla^2f_j(\xi_j)s_{j,\ell},s_{j,\ell}\rangle+\phi_j(x_{j,\ell})-\phi_j(x_{j,\ell}+s_{j,\ell})
\end{aligned}
\end{equation*}
for some $\xi_j\in[x_{j,\ell},x_{j,\ell}+s_{j,\ell}]$ and also that
\begin{equation*}
\begin{aligned}
  &m_{j,\ell}(x_{j,\ell})-m_{j,\ell}(x_{j,\ell}+s_{j,\ell})\\ 
  &\quad=f_{j-1,0}(x_{j-1,0})-f_{j-1,0}(x_{j-1,0}+s_{j-1})+\phi_j(x_{j,\ell})-\phi_j(x_{j,\ell}+P_{j-1}^js_{j-1})\\
  &\qquad-\langle R_{j}^{j-1}\nabla f_j(x_{j,\ell})-\nabla f_{j-1}(x_{j-1,0}),s_{j-1}\rangle,
\end{aligned}
\end{equation*}
which gives us
\begin{equation*}
\begin{aligned}
   |\ared_{j,\ell}&-\pred_{j,\ell}|
   = \Bigl|
        L_j(x_{j,\ell}) - L_j(x_{j,\ell}+s_{j,\ell})
        - \bigl(m_{j,\ell}(x_{j,\ell}) - m_{j,\ell}(x_{j,\ell}+s_{j,\ell})\bigr)
      \Bigr| \\[4pt]
   &= \Bigl|
        -\tfrac{1}{2}\langle \nabla^{2} f_j(\xi_j)s_{j,\ell},\, s_{j,\ell}\rangle 
        -\, f_{j-1,0}(x_{j-1,0})
        \;+\; f_{j-1,0}(x_{j-1,0}+s_{j-1}) \\[-2pt]
   &\qquad\;
        -\, \langle \nabla f_{j-1}(x_{j-1,0}),\, s_{j-1}\rangle
      \Bigr| \\[4pt]
   &\le \kappa_H\|s_{j,\ell}\|^{2}
     \;\le\; \frac{\kappa_H}{\kappa_{\sigma}}\,\Delta_{j,\ell}^{2}.
\end{aligned}
\end{equation*}
but since $(j,\ell)\in\mathcal{V}(i,k)$,  $\kappa_{\epsilon}<1$, and $j-\zeta(i,k)\leq r$, we have that 
\begin{equation*}
    \delta_{j,\ell}\geq \frac{1}{2}\kappa_{\rm fcd }\kappa_{\rm stop}^r\kappa_{\epsilon}^{j-\zeta(i,k)}h_{i,k}\Delta_{j,\ell}\geq \frac{1}{2}\kappa_{\rm fcd }\kappa_{\rm stop}^r\kappa_{\epsilon}^rh_{i,k}\Delta_{j,\ell}>0.
\end{equation*}
Combining the above inequalities and \eqref{Deltaup} yields
$
    |\rho_{j,\ell}-1|\leq1-\eta_2,
$
which implies that $\rho_{j,\ell}\geq\eta_2$. Iteration $(j,\ell)$ is thus very successful, which completes the proof of item 2.,
where, in the last step, we used \eqref{46} and \eqref{427}. 

For item 4, we start by showing that the total decrease in $L_j$ is bounded above by some multiple of $h_{i,k}$ and $\Delta_{j+1,q}$. Note that 
\begin{equation*}
\begin{aligned}
    L_j(x_{j,0}+s_{j,\min})&=f_{j,0}(x_{j,0})+\langle\nabla f_{j,0}(x_{j,0}), s_{j,\min}\rangle\\
    &\qquad+\frac{1}{2}\langle \nabla^2 f_{j,0}(\xi_j)s_{j,\min},s_{j,\min}\rangle+\phi_j(x_{j,0}+s_{j,\min}),
\end{aligned}
\end{equation*}
for some $\xi_j\in[x_{j,0},x_{j,0}+s_{j,\min}]$, where we have defined 
$s_{j,\min}\coloneqq
\argmin\{L_j(x_{j,0}+s_j):\|s_j\|_j\leq\Delta_{j+1,q}\}$.
Note that $F_i$ satisfies Assumption \ref{passump}, i.e. $F_i(x)$ is bounded below for any $x\in\rr^{n_i}$, let $F_i^{\rm low}$ be the lower bound.
Hence, we get that, for all $s_j$ such that $\|s_j\|_j\leq\Delta_{j+1,q}$,
\begin{equation*}
\begin{aligned}
    &L_j(x_{j,0}) - L_j(x_{j,0}+s_j)\le L_j(x_{j,0}) - L_j(x_{j,0}+s_{j,\min})\\[2pt]
    &\quad\le F_j(x_{j,0}) - F_j^{\mathrm{\rm low}}
      + f_{j,0}(x_{j,0}+s_{j,\min}) - f_{j,0}(x_{j,0}) \\[-2pt]
    &\qquad\qquad
      - \langle \nabla f_{j,0}(x_{j,0}),\, s_{j,\min} \rangle
      - \tfrac{1}{2}\langle \nabla^{2}f_{j,0}(\xi_j)s_{j,\min},\, s_{j,\min}\rangle\\[2pt]
    &\quad\le F_j(x_{j,0}) - F_j^{\mathrm{low}}
      + \tfrac{\kappa_H}{2}\|s_{j,\min}\|^{2}
    \le F_j(x_{j,0}) - F_j^{\mathrm{low}}
      + \tfrac{\kappa_H}{2\kappa_{\sigma}}\Delta_{j+1,q}^{2}.
\end{aligned}
\end{equation*}
This, \eqref{430}, $h_{i,k}\geq\epsilon_{i}^h$, $\Delta_{j+1,q}\geq\Delta_{\min}^s$, and $\kappa_{\epsilon}<1$, yield
$$\ell\leq\frac{F_j(x_{j,0})-F_j^{\rm low}}{\frac{1}{2}\kappa_{\rm fcd }\kappa_{\rm stop}\kappa_{\epsilon}^{r}\epsilon_i^h\Delta_{\min}^s}+\frac{\kappa_H}{2\kappa_{\sigma}\epsilon_i^h},$$
which completes the proof of item 4.
\end{proof}

This theorem yields results analogous to \cite[Corollary 4.5, Lemma 4.6, Lemma 4.7, and Theorem 4.8]{gst}. As the proofs follow from \cite{gst}, we omit the details here. These results guarantee the finiteness of the recursion at iteration $(i,k)$ when $\Delta_{i,k}$ is sufficiently small, and they also ensure that each minimization sequence contains at least one successful iteration (as in Corollary~\ref{eachonesuccess}). Further, the trust-region radii are bounded away from zero by a level-independent factor.

Next it is important to show that the algorithm is well defined in that all recursions are finite. The proof is analogous to ~\cite[Theorem 4.9]{gst}. For the convenience for the reader, we present some proof details .
\begin{theorem}\label{T49}
    The number of iterations at each level is finite. Moreover, there exists $\kappa_L\in(0,1)$ such that, for every minimization sequence at level $i=0,\,\dots,\,r$,
    \begin{equation*}
        L_i(x_{i,0})-L_i(x_{i,p+1})\geq\tau_{i,p}\eta_{1}^{i+1}\kappa_L,
    \end{equation*}
    where $\tau_{i,p}$ is the total number of successful iterations in $\bigcup_{\ell=0}^{p}\mathcal{T}(i,\ell)$.
\end{theorem}
\begin{proof}
The proof is analogous to that of~\cite[Theorem~4.9]{gst}, replacing $h_i$, $\kappa_{\mathrm{red}}$, and $\epsilon_{\min}^g$ with $L_i$, $\kappa_{\mathrm{fcd}}$, and $\epsilon_{\min}^h$ respectively. In this way, we obtain an inequality corresponding to \cite[(4.55)]{gst}, which we restate below in our notation.
\begin{equation}\label{455}
    L_i(x_{i,0})-L_i(x_{i,p+1})=\sum_{\ell=0}^{p}{}^{(S)}[L_i(x_{i,\ell})-L_i(x_{i,\ell+1})]\geq\tau_{i,p}\eta_{1}^{i+1}\kappa_L,
\end{equation}
with 
$\kappa_L\coloneqq\kappa_{\mathrm{fcd}}\epsilon_{\min}^h\min\left\{{\epsilon_{\min}^h}{\kappa_H^{-1}},\Delta_{\min}\right\}\in(0,1)$,
for the minimization sequence including iteration $(i,\ell)$. If $i=r$, $L_i=F$ is bounded below by Assumption \ref{passump}, and \eqref{455} imposes that the number of successful iterations in this sequence must again be finite. The same conclusion holds if $i<r$, since $L_i$ is bounded below on the set $\{x\in\rr^{n_i}\,|\,\|x-x_{i,0}\|_i\leq\Delta_{r,k}\}$, which contains $x_{i,p+1}$ because of \cite[Lemma 4.1]{gst} and \eqref{Deltamono}. As for level 0, we conclude that the sequence is finite. Moreover, the same holds for every minimization sequence at level $i$, and the induction is completed.
\end{proof}
This theorem yields the desired worst-case complexity bound, which we state next.
\begin{theorem}
    Assume that there is $F_{\rm low}\in\mathbb{R}$ such that $L_r(x)=F_r(x)\geq F_{\rm low}$ for every $x\in\rr^n$. Then Algorithm 2 needs at most $\mathcal{O}((\epsilon_{\min}^h)^{-2})$ successful Taylor iterations at any level to obtain an iterate $x_{r,k}$ such that $h_{r,k}\leq\epsilon_{r}^h$.
\end{theorem}

\begin{corollary}
    Assume that Algorithm 2 is called at the uppermost level with $\epsilon_{r}^h=0$. Then 
    $\liminf_{k\to\infty} h_{r,k}=0$.
\end{corollary}
\begin{proof}
This proof is analogous to ~\cite[Corollary 4.11]{gst} and follows from Theorem \ref{T49}.
\end{proof}

\end{document}